\documentclass[a4paper,11pt]{amsart}

\usepackage[utf8]{inputenc}		
\usepackage[T1]{fontenc}
\usepackage[english]{babel}

\usepackage{amsfonts}			
\usepackage{amsmath}
\usepackage{amssymb}
\usepackage{amsthm}

\usepackage[normalem]{ulem}

\usepackage{graphicx}			
\usepackage{tikz}
\usetikzlibrary{cd}

\usepackage{hyperref}			
\hypersetup{colorlinks=false, urlcolor=black, linkcolor=black}
\usepackage{cleveref}

\usepackage{enumitem}			

\usepackage{color}				
\usepackage{float}

\usepackage{thmtools}
\usepackage{thm-restate}

\usepackage{mathrsfs}

\newcommand{\Z}{\mathbb{Z}}						
\newcommand{\R}{\mathbb{R}}						

\renewcommand{\S}{\mathbb{S}}					
\newcommand{\B}{\mathbb{B}}

\newcommand{\eps}{\varepsilon}					

\newcommand{\dd}								
{\mathop{}\!\mathrm{d}}						
\newcommand{\ddn}[1]							
{\mathop{}\!\mathrm{d^{#1}}}

\newcommand{\abs}[1]							
{\left| #1 \right|}
\newcommand{\smallabs}[1]						
{\lvert #1 \rvert}	
\newcommand{\norm}[1]							
{\left\lVert #1 \right\rVert}	
\newcommand{\smallnorm}[1]						
{\lVert #1 \rVert}						
\newcommand{\ip}[2]								
{\left< #1 , #2 \right>}

\DeclareMathOperator*{\osc}{osc}		
\DeclareMathOperator*{\esssup}{ess\,sup}		
\DeclareMathOperator*{\essinf}{ess\,inf}
\DeclareMathOperator*{\esslimsup}{ess\,lim\,sup}
\DeclareMathOperator*{\essliminf}{ess\,lim\,inf}
\DeclareMathOperator*{\essosc}{ess\,osc}		
\DeclareMathOperator{\divg}{div}

\DeclareMathOperator{\dist}{dist}

\newcommand{\loc}{\mathrm{loc}}

\newcommand{\cH}{\mathcal{H}}

\renewcommand{\phi}{\varphi}

\def\XXint#1#2#3{{\setbox0=\hbox{$#1{#2#3}{\int}$}
		\vcenter{\hbox{$#2#3$}}\kern-.5\wd0}}

\newtheorem{thm}{Theorem}[section]{\bf}{\it}
\newtheorem{lemma}[thm]{Lemma}
\newtheorem{prop}[thm]{Proposition}
\newtheorem{cor}[thm]{Corollary}

{\bf}{\it}

{\bf}{\it}

{\bf}{\it}

{\bf}{\it}

\theoremstyle{definition}
\newtheorem{defn}[thm]{Definition}

\theoremstyle{remark}
\newtheorem{rem}[thm]{Remark}

\numberwithin{equation}{section}

\begin{document}
	
	\title{Continuity for Sobolev mappings with null Lagrangian bounds}

	\author[I. Kangasniemi]{Ilmari Kangasniemi}
	\address{Department of Mathematics and Statistics, 
		P.O. Box 35 (MaD), FI-40014 University of Jyväskylä, Finland.
	}
	\email{ilmari.k.kangasniemi@jyu.fi}
	
	\author[J. Onninen]{Jani Onninen}
	\address{Department of Mathematics, Syracuse University, Syracuse,
		NY 13244, USA and  Department of Mathematics and Statistics, P.O.Box 35 (MaD) FI-40014 University of Jyv\"askyl\"a, Finland
	}
	\email{jkonnine@syr.edu}
	
	\date{\today}
	\thanks{J. Onninen was supported by the National Science Foundation grant DMS-2453853. During the early investigative phase of this work, I.\ Kangasniemi was supported by the National Science Foundation grant DMS-2247469.}
	
	\begin{abstract}	
		We prove the continuity of Sobolev functions $\varphi \in W^{1,n}_{\mathrm{loc}}(\Omega)$, $\Omega \subset \mathbb{R}^n$, that satisfy \[
			\lvert\nabla \varphi(x)\rvert^n \le K(x)\bigl(\langle \nabla \varphi(x), \xi(x)\rangle + A(x)\bigr),
		\] 
		where $\xi \in L_{\mathrm{loc}}^{n/(n-1)}(\Omega, \mathbb{R}^n)$ is weakly divergence-free, and $K \in L^p_{\mathrm{loc}} (\Omega)$, $A \in L^q_{\mathrm{loc}} (\Omega)$ are non-negative with $p^{-1}+q^{-1}<1$. The result is applicable to a broad class of differential inequalities of null Lagrangian type. As our principal application, we obtain a sharp continuity theorem for $f \in W^{1,n}_{\mathrm{loc}} (\Omega, \mathbb{R}^n)$ satisfying the distortion inequality with defect $\lvert Df(x)\rvert^n \le K(x)\det Df(x) + \Sigma(x)$; this result is new even in the planar case, and closes a significant gap between existing methods and known counterexamples. The proof relies on an overlooked Sobolev-type inequality formulated in terms of measures of superlevel sets.
	\end{abstract}

	\maketitle
	
	\section{Introduction}	
	Let $\Omega \subset \R^n$ be open. We study real-valued functions $\phi \in W^{1,n}_\loc(\Omega)$ satisfying the pointwise inequality
	\begin{equation}\label{eq:divergence_distortion_ineq}
		\abs{\nabla \phi(x)}^n \le K(x) \bigl( \ip{\nabla \phi(x)}{\xi(x)} + A(x) \bigr) \qquad 
		\text{for a.e.\ } x \in \Omega,
	\end{equation}
	where $K,A \colon \Omega \to [0,\infty)$ are measurable and $\xi \in L^{n/(n-1)}_{\loc}(\Omega,\R^n)$ is weakly divergence free. The following theorem provides the precise regularity threshold which ensures the continuity of $\varphi$.
	
	\begin{restatable}{thm}{continuitygeneral}\label{thm:continuity_general}
		Let $n \ge 2$, let $\Omega \subset \R^n$ be open, let $K, A \colon \Omega \to [0, \infty)$ be measurable, and let $\xi \in L^{n/(n-1)}_\loc(\Omega, \R^n)$ with $\divg \xi = 0$ weakly. Suppose that $\phi \in W^{1,n}_\loc(\Omega)$ satisfies \eqref{eq:divergence_distortion_ineq}.
		If there exist $p, q \in [1, \infty]$ such that
		\[
			K \in L^p_\loc(\Omega) \quad \text{and} \quad  A \in L^q_\loc(\Omega) \quad \text{with} \quad \frac{1}{p} + \frac{1}{q} < 1,
		\]
		then $\phi$ admits a continuous representative. Moreover, for each $x_0 \in \Omega$, the local modulus of continuity of $\phi$ satisfies  $\omega_\phi(r;x_0)=O(\log^{-1/n}(1/r))$ as $r \to 0$. 
	\end{restatable}
	
	This result is sharp. Indeed, as shown in \cite[Theorem 1.11]{Dolezalova-Kangasniemi-Onninen_MGFD-cont}, for all $p, q \in [1, \infty]$ with $1/p + 1/q \ge 1$, there exist $\varphi \in W^{1,2}_\loc(\B^2)$, $\xi \in L^2_\loc(\B^2)$ with $\divg \xi = 0$ weakly, $K \in L^p_\loc(\B^2)$, and $A \in L^q_\loc(\B^2)$  such that \eqref{eq:divergence_distortion_ineq} holds while $\varphi$ fails to be continuous.

	\subsection{Null Lagrangians and the divergence structure} 
	Inequalities of the form \eqref{eq:divergence_distortion_ineq} arise naturally in the study of mappings $f \in W^{1,n}_{\loc}(\Omega,\R^m)$  satisfying a null Lagrangian inequality
	\begin{equation}\label{eq:null_lagrangian_inequality}
		|Df(x)|^n \le K(x)L(x,Df(x)) \qquad \text{for a.e. } x \in \Omega,
	\end{equation}
	where $\abs{Df(x)}$ denotes the operator norm  of $Df(x) \colon \R^n \to \R^m$,  and $L \colon \Omega \times \R^{n \times m} \to \R$ is a $C^1$-smooth null Lagrangian.  Recall that a mapping  $L \colon \Omega \times \R^m \times \R^{n \times m} \to \R$  is called a  \emph{null Lagrangian}  if 
	\[
		\int_\Omega L(x, (f+\eta)(x), D (f+\eta)(x)) \dd m_n(x)
		= \int_\Omega L(x, f(x), D f(x)) \dd m_n(x)
	\]
	for all $f \in C^1(\overline{\Omega}, \R^m)$ and $\eta \in C^\infty_c(\Omega, \R^m)$, where $m_n$ stands for the $n$-dimensional Lebesgue measure. The special case considered in \eqref{eq:null_lagrangian_inequality} corresponds to  $L$  being independent of the values of $f$

	The theory of null Lagrangians originates in  the classical calculus of variations and field theory, beginning with Carathéodory~\cite{Caratheodory1929}, Weyl~\cite{Weyl1935}, and Rund~\cite{Rund1966}; see also Morrey~\cite{Morrey1966}. In nonlinear elasticity, they play a fundamental role in variational formulations and stability theory; see e.g. the works of Ball~\cite{Ball1977}, Ball, Currie, and Olver~\cite{Ball-Currie-Olver1981}, Davini and Parry~\cite{Davini-Parry1988}, Edelen and Lagoudas~\cite{Edelen-Lagoudas1986}, Ericksen~\cite{Ericksen1962}, and Sivaloganathan~\cite{Sivaloganathan1988}. A structural characterization of null Lagrangians, due to Olver and Sivaloganathan~\cite{Olver-Sivaloganathan_nullLagrangians}, asserts that  every $C^1$-smooth null Lagrangian can be represented as 
	\[
		L(x, f(x), Df(x)) = L_0(x, f(x)) + \sum_I L_I(x, f(x)) M_I(Df(x)),
	\] 
	where $M_I(Df(x))$ are minors of $Df(x)$.
	
	If $L$ is independent of $f$ as in \eqref{eq:null_lagrangian_inequality}, it is  hence  an affine combination of minors of $Df$ with coefficients depending on $x$.  Consequently, if  $f \in W^{1,n}_{\loc }(\Omega, \R^m)$ 
	satisfies \eqref{eq:null_lagrangian_inequality} for such an $L$, then each coordinate function  $\varphi$ of $f$ satisfies an inequality of the form \eqref{eq:divergence_distortion_ineq}. The associated divergence-free vector field  $\xi$ arises from the terms where the minor $M_I(Df(x))$ involves $\nabla \phi$, and the function $A$ collects the remaining terms. Thus, if the resulting $A$ has sufficiently good regularity, then Theorem~\ref{thm:continuity_general} yields that $\phi$ is continuous. This approach to proving continuity is also valid for $W^{1,n}_\loc$-solutions of 
	\[
		|Df(x)|^n \le K(x)L(x, f(x), Df(x)) \qquad \text{for a.e. } x \in \Omega,
	\]
	whenever $L(x,f,Df) \le L(x,Df) + e^{|f|^\alpha} $ with $\alpha < n/(n-1)$, since in this case  Trudinger's theorem ensures that the exponential factor $e^{|f|^\alpha}$ can be absorbed into $A$.
 	\enlargethispage{-\baselineskip}

	\subsection{Distortion inequalities with defects} 
	A widely studied case of \eqref{eq:null_lagrangian_inequality} is when $m = n$, $K \ge 1$, and $L(x, Df(x)) = J_f(x)$ is the Jacobian determinant of $f$. This leads to the classical theory of \emph{mappings of finite distortion}. When $K$ is bounded, one obtains the intermediate class of \emph{mappings of bounded distortion}, or \emph{$K$-quasiregular maps}. The foundations of the higher dimensional theory of quasiregular maps were laid by Reshetnyak, who among other results established the local H\"older continuity of such mappings with sharp exponent  $1/K$ \cite{Reshetnyak_continuity, Reshetnyak-book}, generalizing a prior planar result by Morrey~\cite{Morrey}. Allowing $K$ to be unbounded leads to a substantially larger class of mappings which also arise naturally in nonlinear elasticity~\cite{Antman_Elasticity-book, Ball_nonlinear-elasticity, Ciarlet_Elasticity-book}. The continuity of $W^{1,n}_\loc$-regular mappings of finite distortion is valid with no integrability assumptions on $K$; see \cite{Godshtein-Vodopyanov_Continuity, Iwaniec-Koskela-Onninen-Invent}.

	Our original motivation behind Theorem \ref{thm:continuity_general} is to study mappings $f \in W^{1,n}_{\loc}(\Omega, \R^n)$ that satisfy the \emph{distortion inequality with defect}
	\begin{equation}\label{eq:generalized_finite_distortion}
    	\abs{Df(x)}^n \le K(x) J_f(x) + \Sigma(x),
	\end{equation}
	where $K \colon \Omega \to [1,\infty)$ and $\Sigma \colon \Omega \to [0,\infty)$ are measurable functions. With the introduction of the defect $\Sigma$, the continuity of solutions to~\eqref{eq:generalized_finite_distortion} becomes a delicate problem, previously understood essentially only in the case of bounded distortion $K \in L^\infty(\Omega)$. This setting already arises in the classical theory of elliptic partial differential equations \cite[Chapter~12]{Gilbarg-Trudinger_Book}. For bounded $K$ and $\Sigma$, H\"older regularity in the plane was established in the seminal works of Nirenberg~\cite{Nirenberg_Holder}, Finn and Serrin~\cite{Finn_Serrin}, and Hartman~\cite{Hartman}. These techniques also underlie Simon’s estimates for solutions between surfaces, with applications to equations of mean curvature type~\cite{Simon_QRGauss}.
	
	If $K \in L^\infty(\Omega)$ and $\Sigma \in L^p_{\loc}(\Omega)$ for some $p>1$, then every solution of~\eqref{eq:generalized_finite_distortion} admits a H\"older continuous representative. In the plane, this was proved by Astala, Iwaniec, and Martin \cite[Theorem~8.5.1]{Astala-Iwaniec-Martin_Book}; in higher dimensions, the result follows from prior work of the authors \cite[Section~3]{Kangasniemi-Onninen_Heterogeneous}. The planar argument relies on potential-theoretic methods from complex analysis, while the higher-dimensional proof proceeds along the lines of Morrey~\cite{Morrey} and Reshetnyak~\cite{Reshetnyak_continuity}.  These methods extend slightly beyond the case $K \in L^\infty(\Omega)$, as they remain valid under the exponential integrability condition  $\exp(\lambda K)\in L^1_{\loc}(\Omega)$ where $\lambda>n+1$; see the work of Dole\v{z}alov\'a and the authors in \cite{Dolezalova-Kangasniemi-Onninen_MGFD-cont}. 
	
	The limitations of the prior methods already become apparent in the case $\Sigma \in L^\infty(\Omega)$. Indeed, for bounded~$\Sigma$, the previously known continuity results for  planar solutions to \eqref{eq:generalized_finite_distortion} can be summarized as follows:
	\begin{itemize}
		\item If $\Sigma \equiv 0$, continuity holds for arbitrary $K$.
		\item If $\Sigma \equiv 1$, continuity fails for $K\in L_{\loc}^1(\Omega)$, is unknown for $K\in L_{\loc}^p(\Omega)$ with $1<p<\infty$, and holds if $\exp(\lambda K)\in L_{\loc}^1(\Omega)$ for $\lambda>3$.
	\end{itemize}
	A discontinuous planar solution to \eqref{eq:generalized_finite_distortion} with $K \in L^1(\Omega)$ and $\Sigma \in L^\infty(\Omega)$ was constructed in \cite{Dolezalova-Kangasniemi-Onninen_MGFD-cont}. More generally, for any exponents $p,q \in [1,\infty]$ with $p^{-1}+q^{-1}\ge 1$, there exists a discontinuous solution whenever $K \in L_{\loc}^p(\Omega)$ and $\Sigma/K \in L_{\loc}^q(\Omega)$; see \cite[Theorem 1.11]{Dolezalova-Kangasniemi-Onninen_MGFD-cont}. Despite the considerable gap between the known counterexamples and the reach of available techniques, it was conjectured in \cite[Conjecture 1.10]{Dolezalova-Kangasniemi-Onninen_MGFD-cont} that these integrability conditions are in fact sharp.
	
	With Theorem \ref{thm:continuity_general}, we are able to confirm this conjecture in full generality.
	\enlargethispage{-\baselineskip}
	
	\begin{restatable}{thm}{continuity}\label{thm:continuity}
		Let $n \ge 2$, let $\Omega \subset \R^n$ be open, and let $K \colon \Omega \to [1, \infty)$, $\Sigma \colon \Omega \to [0, \infty]$ be measurable. Suppose that $f \in W^{1,n}_\loc(\Omega, \R^n)$ satisfies the distortion inequality with defect \eqref{eq:generalized_finite_distortion}. If there exist $p, q \in [1, \infty]$ such that
		\[
			K \in L^p_\loc(\Omega)
			\quad \text{and} \quad
			\frac{\Sigma}{K} \in L^q_\loc(\Omega)
			\quad \text{with} \quad
			\frac{1}{p} + \frac{1}{q} < 1,
		\] 
		then $f$ has a continuous representative, with local modulus of continuity satisfying $\omega_f(r; x_0) = O(\log^{-1/n}(1/r))$ as $r \to 0$ at every $x_0 \in \Omega$.
	\end{restatable}

	In the only previously known case $p = \infty$, the solutions of \eqref{eq:generalized_finite_distortion} are H\"older continuous with critical H\"older exponent $\gamma = \min \left.(\norm{K}^{-1}_{L^\infty(\Omega)}, 1 - q^{-1}\right.)$; see \cite[Theorem 1.1]{Kangasniemi-Onninen_Heterogeneous} for details. When $p<\infty$, our estimate of the local modulus of continuity is already sharp in the class of mappings of finite distortion, as the map $f(x)= \log^{-1/n}\!\bigl(\abs{x}^{-1}\bigr) \abs{x}^{-1} x$ on the unit ball $\B^n \subset \R^n$  satisfies the  inequality~\eqref{eq:generalized_finite_distortion} with $\Sigma\equiv 0$ and $K \in L^{p}(\mathbb B^n)$ for all  ${p}\in [1, \infty)$. 
	
	\subsection{Quasiregular values and values of finite distortion} The inequality \eqref{eq:generalized_finite_distortion} is also closely related to the  recently developed theory of quasiregular values \cite{Kangasniemi-Onninen_Heterogeneous, Kangasniemi-Onninen_Heterogeneous-corrigendum, Kangasniemi-Onninen_1ptReshetnyak,Kangasniemi-Onninen_Picard,Kangasniemi-Onninen_Rescaling}. Given $K \in [1,\infty) $ and  $\Sigma \in L^1_{\loc}(\Omega) $, a mapping  $f \in W^{1,n}_{\loc}(\Omega,\R^n) $ is said to have a \emph{$(K, \Sigma)$-quasiregular value} at  $y_0 \in \R^n $ if
	\begin{equation}\label{eq:QR-value}
		\abs{Df(x)}^n \le K J_f(x) + |f(x)-y_0|^n \Sigma(x) \qquad \text{for a.e. } x \in \Omega.
	\end{equation}
 	This definition provides a pointwise notion of quasiregularity, and leads to single–value versions of classical results for quasiregular mappings, including Reshetnyak- and Rickman--Picard-type theorems. A core feature of the theory of quasiregular values is that if $f \in W^{1,n}_\loc(\Omega, \R^n)$ satisfies \eqref{eq:QR-value}, then the spherical logarithm $F \colon \Omega \to \R \times \S^{n-1}$ of $f - y_0$ obeys  $\abs{DF(x)}^n \le K J_F(x) + \Sigma(x)$ whenever applicable; see e.g.\ \cite[Section 4.1]{Kangasniemi-Onninen_Picard}. For this reason, arguments that can prove regularity properties of solutions of \eqref{eq:generalized_finite_distortion} are a central tool in the theory of quasiregular values. 

	Motivated by the theory of quasiregular values, we introduce the broader concept of \emph{values of finite distortion}. Specifically, if $K \colon \Omega \to [1, \infty)$ and $\Sigma \colon \Omega \to [0, \infty)$ are measurable, we say that $f$ has a \emph{value of finite distortion} with data $(K, \Sigma)$ at $y_0 \in \R^n$ if 
	\begin{equation}\label{eq:value_of_finite_distortion}
		\abs{Df(x)}^n \le K(x) J_f(x) + \abs{f(x) - y_0}^n \Sigma(x) \qquad \text{for a.e. } x \in \Omega.
	\end{equation}
	Theorem \ref{thm:continuity} then also yields the following continuity result for mappings which have a value of finite distortion. 		
	
	\begin{restatable}{cor}{continuityVFD}\label{cor:continuity_VFD}
		Let $n \ge 2$, let $\Omega \subset \R^n$ be open, let $y_0 \in \R^n$, and let $K \colon \Omega \to [1, \infty)$, $\Sigma \colon \Omega \to [0, \infty]$ be measurable. Suppose that $f \in W^{1,n}_\loc(\Omega, \R^n)$ has a value of finite distortion with data $(K, \Sigma)$ at $y_0$. If there exist $p, q \in [1, \infty]$ such that
		\[
			K \in L^p_\loc(\Omega)
			\quad \text{and} \quad
			\frac{\Sigma}{K} \in L^q_\loc(\Omega)
			\quad \text{with} \quad
			\frac{1}{p} + \frac{1}{q} < 1,
		\] 
		then $f$ has a continuous representative, with local modulus of continuity satisfying $\omega_f(r; x_0) = O(\log^{-1/n}(1/r))$ as $r \to 0$ at every $x_0 \in \Omega$.
	\end{restatable}
	
	Note that the same counterexamples as in \cite[Theorem 1.11]{Dolezalova-Kangasniemi-Onninen_MGFD-cont} show that  Corollary \ref{cor:continuity_VFD} is sharp on the $L^p$-scale.
	
	\subsection{The strategy of the proof}
	
	A powerful tool in the analysis of continuity properties of functions is the notion of monotonicity, introduced by  Lebesgue~\cite{Lebesgue_1907}. A continuous function $\varphi \colon \Omega \to \R$ is called \emph{monotone} if, for every open ball $B$ that is compactly contained in $\Omega$, the oscillation of $\varphi$ is attained on the boundary:
	\[
	\osc_{B} \varphi = \osc_{\partial B} \varphi. \]
	Equivalently, a monotone function $\varphi $ satisfies both the maximum and minimum principles:
	\[
	\inf_{B}\varphi  =  \min_{\partial B}\varphi 
	\leq  \sup_{B}\varphi =  \max_{\partial B}\varphi .\]
	The relevance of this principle for elliptic PDEs is immediate. 
	
	When dealing with very weak solutions to differential inequalities, such as mappings of finite distortion satisfying \eqref{eq:generalized_finite_distortion}   with $\Sigma \equiv 0$, one  needs to adopt a notion of so-called
	weakly monotone functions, due to  Manfredi~\cite{Manfredi_weakly}. Here, a Sobolev function $\phi \colon \Omega \to \R$ is called \emph{weakly monotone} if, for every ball $B$ that is compactly contained in $\Omega$ and all $M, m \in \R$ such that
	\[
	(\phi|_{B} - M)^+,  (m - \phi|_{B})^+ \in W^{1,1}_0(B),
	\]
	we have $(\phi\vert_{B} - M)^+ \equiv 0 \equiv (m - \phi\vert_B)^+$. A key observation is that a weakly monotone $W^{1,n}_\loc$-function always has a continuous Sobolev representative; see e.g.\ \cite[Theorem 2.21]{Hencl-Koskela-book}. The standard proof of the case $\Sigma \equiv 0$ of Theorem \ref{thm:continuity} is hence by showing that all coordinate functions $f_i$ of $f$ are weakly monotone; see e.g.\ \cite[Section 4]{Iwaniec-Koskela-Onninen-Invent} or  \cite[Theorem 2.17]{Hencl-Koskela-book} for the proof. 
	
	If $\Sigma \not\equiv 0$, however, coordinate functions of solutions to \eqref{eq:generalized_finite_distortion} need not be weakly monotone. This motivates us to introduce a further quantitative weakening of monotonicity: whereas a continuous monotone function admits no new extrema in the interior of a ball, we measure the failure of this principle by a H\"older-type defect of the form 
	\[ \osc_{B_r} \varphi \leq \osc_{\partial B_r} \varphi +Cr^\alpha \, .  \]
	For Sobolev functions, which are not necessarily continuous, this leads to the following  definition.
	
	\begin{defn}\label{def:weak_weak_monotone}
		Let $\Omega \subset \R^n$ be open and let $\phi \in W^{1,1}_\loc(\Omega)$. We say that $\phi$ is \emph{$\alpha$-almost weakly monotone} if there exist constants $C, \alpha > 0$ such that, whenever $B_r \subset \Omega$ is a ball of radius $r$ that is compactly contained in $\Omega$ and $M, m \in \R$ are such that
		\[
		(\varphi \vert_{B_r} - M)^+ \in W^{1,1}_0(B_r)
		\quad \text{and} \quad
		(m - \varphi \vert_{B_r})^+ \in W^{1,1}_0(B_r),
		\]
		then we have
		\[
		\norm{(\varphi \vert_{B_r} - M)^+}_{L^\infty(B_r)} \le Cr^{\alpha}
		\quad \text{and} \quad
		\norm{(m - \varphi \vert_{B_r})^+}_{L^\infty(B_r)} \le Cr^{\alpha}.
		\]
	\end{defn}
	
	As with weak monotonicity, $\alpha$-almost weakly monotone $W^{1,n}$-functions  have continuous representatives.
	
	\begin{prop}\label{prop:weak_weak_mono_cont}
		Let $n \ge 2$, let $\Omega \subset \R^n$ be open, and let $\phi \in W^{1,n}_\loc(\Omega)$. If $\phi$ is $\alpha$-almost weakly monotone for any $\alpha > 0$, then $\phi$ has a continuous representative, with local modulus of continuity satisfying $\omega_\phi(r; x_0) = O(\log^{-1/n}(1/r))$ as $r \to 0$ at every $x_0 \in \Omega$.
	\end{prop}
	
	The principal challenge is to show that the function $\phi$ in Theorem~\ref{thm:continuity_general} is almost weakly monotone. We achieve this by proving the following estimate.
	
	\begin{restatable}{thm}{Linftyestimate}\label{thm:weakly_vanishing_L_infinity_estimate}
		Let $n \ge 2$, let $\Omega \subset \R^n$ be open and bounded, and let $\phi \in W^{1,n}_0(\Omega)$ satisfy \eqref{eq:divergence_distortion_ineq}, where $K, A \colon \Omega \to [0, \infty)$ are measurable functions and $\xi \in L^{n/(n-1)}(\Omega, \R^n)$ with $\divg \xi = 0$ weakly. If $p, q \in [1, \infty]$ are such that $K \in L^p(\Omega)$ and $A \in L^q(\Omega)$ with $p^{-1} + q^{-1} < 1$, then
		\[
			\norm{\phi}_{L^\infty(\Omega)}^n 
			\le C(n,p,q) \norm{K}_{L^p(\Omega)} \norm{A}_{L^q(\Omega)}
			[m_n(\Omega)]^{1 - p^{-1} - q^{-1}}.
		\]
	\end{restatable}
	
	Notably, the critical tool that sets the proof of Theorem \ref{thm:weakly_vanishing_L_infinity_estimate} into motion is a non-standard Sobolev inequality that we call a \emph{superlevel Sobolev inequality}. Although the proof is not particularly difficult, the inequality appears to have been previously overlooked; we are not aware of any prior formulation in the literature. We expect this inequality to be of independent interest, with potential applications  beyond the present context. In the statement, we use the standard notation $\omega_n := m_n(\B^n)$ for the $n$-dimensional Lebesgue measure of the unit ball in $\R^n$.
	
	\begin{prop}\label{prop:distr_Sobolev_ineq}
		Let $n \ge 2$, let $\Omega \subset \R^n$ be open, and let $\phi \in W^{1,1}_0(\Omega)$ with $\phi \ge 0$ a.e.\ in $\Omega$. Then
		\[
			\norm{\phi}_{L^\infty(\Omega)} \le \frac{1}{n \omega_n^{1/n}} 
			\int_\Omega 
			\frac{\abs{\nabla \phi(x)}}{[m_n(\{ z \in \Omega : \phi(z) \ge \phi(x) \})]^{\frac{n-1}{n}}}
			\dd m_n(x).
		\]
	\end{prop}
	
	For us, the most significant feature of this superlevel Sobolev inequality is that the integrand in the upper bound is of the form $(F \circ \phi) \abs{\nabla \phi}$, where $F$ is a real function. This is critical, since it unlocks strategies based on level set integrals of null Lagrangians, which have been previously used to great effect in the study of quasiregular values; see e.g.\ \cite{Kangasniemi-Onninen_Heterogeneous}, \cite{Kangasniemi-Onninen_1ptReshetnyak}, and \cite{Kangasniemi-Onninen_Picard}.
	
	\subsection{The structure of this article} In Section \ref{sect:prelims}, we recall various necessary preliminary results from the theory of Sobolev functions. In Section \ref{sect:distr_funct}, we provide proofs of several integral estimates involving powers of distribution functions $\mu_\phi^{+}(t) = m_n(\phi^{-1}[t, \infty])$ and $\mu_\phi^{-}(t) = m_n(\phi^{-1}(t, \infty])$, including a proof of Proposition \ref{prop:distr_Sobolev_ineq}. In Section \ref{sect:almost_weak_mono_is_cont}, we prove Proposition \ref{prop:weak_weak_mono_cont}. Finally, in Section \ref{sect:main_result}, we complete the proofs of Theorems \ref{thm:weakly_vanishing_L_infinity_estimate}, \ref{thm:continuity_general}, and \ref{thm:continuity}.
	
	\subsection*{Acknowledgments} The initial ideas of this work were conceived during a visit by J.O.\ to the University of Cincinnati, near the end of I.K.'s appointment there; the authors thank the university for providing an environment conducive to collaboration.  We also thank Tero Kilpeläinen, Pekka Koskela, Pekka Pankka, Nageswari Shanmugalingam, Lubo\v{s} Pick, and Andrea Cianchi for helpful discussions and/or correspondence.
	
	\section{Preliminaries}\label{sect:prelims}
	
	Throughout this article, we use the notation $h^{+}$ and $h^{-}$ for the positive and negative parts of a real valued function $h \colon \Omega \to \R$. That is, $h^{+}, h^{-} \colon \Omega \to [0, \infty)$ are given by 
	\begin{align*}
		h^{+}(x) &= \max(h(x), 0)&\text{and}&&
		h^{-}(x) &= \max(-h(x), 0)
	\end{align*}
	for all $x \in \Omega$. In particular, $h = h^{+} - h^{-}$. The main exception to this are the distribution functions $\mu_\varphi^{+}$ and $\mu_\varphi^{-}$ defined in Section \ref{sect:distr_funct}, for which the superindices `$+$' and `$-$' are unrelated to this convention.
	
	\subsection{Piecewise constant approximations}
	
	We begin by recalling a basic approximation result which we call the Staircase Lemma. It allows one to uniformly approximate a non-decreasing left-continuous function with a staircase of constant functions; see Figure \ref{fig:staircase_lemma} for an illustration.
	
	\begin{figure}[h]\label{fig:staircase_lemma}
		\begin{center}\begin{tikzpicture}[scale=0.9]
			\begin{scope}
				\draw[->] (-0.1, 0) -- (6, 0);
				\draw[->] (0, -0.1) -- (0, 6);
				
				\draw plot [smooth, tension=0.5] coordinates {(0,0) (1, 1) (1.5, 1.15) (2, 2) (2.1, 2.5) (2.4, 3)};
				\draw [-] (2.4, 4) -- (3.5, 4);
				\draw plot [smooth, tension=0.5] coordinates {(3.5,4.3) (3.7, 4.5) (3.8, 4.9) (4.7, 5)
												(5.3, 5.5)};
				\draw[dashed] (5.3,5.5) -- (5.8, 6);
				
				\fill (2.4, 3) circle (1.2pt);
				\fill (3.5, 4) circle (1.2pt);
			\end{scope}
			
			\begin{scope}[shift={(7.5,0)}]
				\draw[->] (-0.1, 0) -- (6, 0);
				\draw[->] (0, -0.1) -- (0, 6);
				
				\draw[dotted] plot [smooth, tension=0.5] coordinates {(0,0) (1, 1) (1.5, 1.15) (2, 2) (2.1, 2.5) (2.4, 3)};
				\draw[dotted] (2.4, 4) -- (3.5, 4);
				\draw[dotted] plot [smooth, tension=0.5] coordinates {(3.5,4.3) (3.7, 4.5) (3.8, 4.9) (4.7, 5)
					(5.3, 5.5)};
				\draw[dotted] (5.3,5.5) -- (5.8, 6);
				
				\draw[-] (0, 0.9) -- (0.9, 0.9);
				\draw[-] (0.9, 1.8) -- (1.9, 1.8);
				\draw[-] (1.9, 2.7) -- (2.2, 2.7);
				\draw[-] (2.2, 3) -- (2.4, 3);
				\draw[-] (2.4, 4.5) -- (3.7, 4.5);
				\draw[-] (3.7, 5.4) -- (5.2, 5.4);
				
				\fill (0, 0) circle (1.2pt);
				\fill (0.9, 0.9) circle (1.2pt);
				\fill (1.9, 1.8) circle (1.2pt);
				\fill (2.2, 2.7) circle (1.2pt);
				\fill (2.4, 3) circle (1.2pt);
				\fill (3.7, 4.5) circle (1.2pt);
				\fill (5.2, 5.4) circle (1.2pt);
				
				\draw[-] (-0.4, 0) -- (-0.4, 0.45) node [anchor=east] {$\eps$} -- (-0.4, 0.9);
				\draw[dashed] (-0.4,0.9) -- (-0.4, 6);
				\draw[-] (-0.5, 0) -- (-0.3, 0);
				\draw[-] (-0.5, 0.9) -- (-0.3, 0.9);
				\draw[-] (-0.5, 1.8) -- (-0.3, 1.8);
				\draw[-] (-0.5, 2.7) -- (-0.3, 2.7);
				\draw[-] (-0.5, 3.6) -- (-0.3, 3.6);
				\draw[-] (-0.5, 4.5) -- (-0.3, 4.5);
				\draw[-] (-0.5, 5.4) -- (-0.3, 5.4);
			\end{scope}
		\end{tikzpicture}\end{center}
		\caption{\small A non-decreasing left-continuous function, along with one of its staircase approximations.}
	\end{figure}
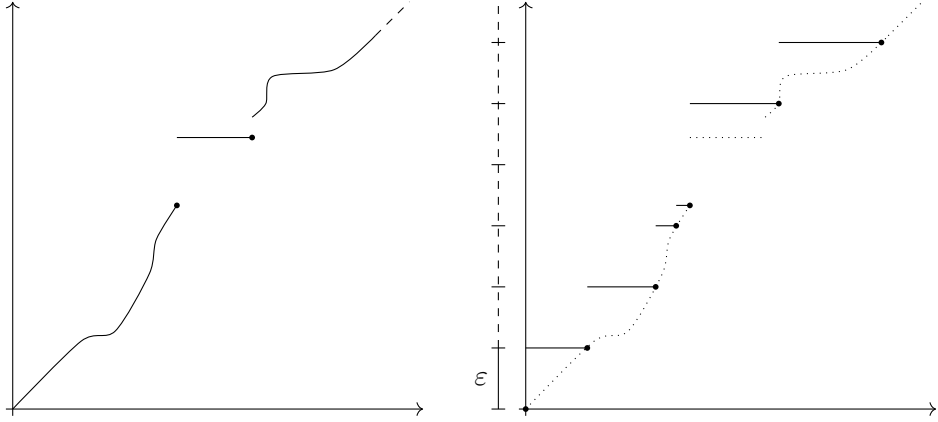
	
	\begin{lemma}[Staircase Lemma]\label{lem:left_cont_approx}
		Let $F \colon [0, \infty] \to [0, \infty]$ be left-continuous and non-decreasing. Let $s := \sup \{ t \in [0, \infty] : F(t) < F(\infty)\}$. If $s > 0$, then for every $\eps > 0$, there exists an increasing sequence $t_0 < t_1 < t_2 < \dots$ of elements of $[0, s)$ for which $t_0 = 0$, $\lim_{i \to \infty} t_i = s$, and for every $i \in \Z_{>0}$ and all $t \in (t_{i-1}, t_i]$, we have $\abs{F(t_i) - F(t)} \le \eps$.
	\end{lemma}
	\begin{proof}
		Fix $\eps > 0$, and let $t_0 = 0$. For every $i \in \Z_{> 0}$, we let
		\[
			t'_i = \sup \{ t \in [0, s] : F(t) \le F(0) + i\eps\},
		\]
		where we additionally denote $t_0' = 0$. Note that the set we are taking a supremum over in the definition of $t_i'$ is non-empty, and that $t_i'$ form a non-decreasing sequence. Since $F$ is left-continuous, it hence follows that $F(t'_i) \le F(0) + i\eps$. Consequently, since $F$ is non-decreasing, we have for all $i \in \Z_{> 0}$ that $F(0) + (i-1)\eps < F(t) \le F(0) + i\eps$ when $t \in (t_{i-1}', t_i']$. In particular, $\abs{F(t) - F(t')} \le \eps$ for all $t, t' \in (t_{i-1}', t_i']$.
		
		Now, we split into two cases. In the first case, we have $t_i' < s$ for all $i$. This is only possible if $F(\infty) = \infty$, since otherwise we would have $F(\infty) \le F(0) + i\eps$ for some $i \in \Z_{> 0}$, at which point $t'_i = s$. We claim that in this case $\lim_{i \to \infty} t_i' = s$. Indeed, otherwise $\lim_{i \to \infty} t_i' < s$, and we may select a $t \in (\lim_{i \to \infty} t_i', s)$; we then have $F(t) > F(0) + i\eps$ for all $i$ implying that $F(t) = \infty$, which contradicts the fact that $F(t) < F(\infty) = \infty$ due to $t < s$. Thus, $\lim_{i \to \infty} t_i' = s$, and the claim follows if we select $t_i$ to be the numbers $t_i'$ with any repeats skipped.
		
		The other case is that $t_k' = s$ for some $k \in \Z_{> 0}$. Suppose that $k$ is the first such index, implying that $s > t_{k-1}'$. Now we again select $t_i$ to be the numbers $t'_i$ with repeats skipped, up to $t'_{k-1}$. Then, for the remaining $t_i$, we select them to be any increasing sequence in $(t_{k-1}', s)$ tending to $s$, and the claim follows.
	\end{proof}
	
	\subsection{Sobolev functions}
	
	Let $\Omega \subset \R^n$ be open, and let $p \in [1, \infty)$. We use $W^{1,p}(\Omega)$ to denote the standard Sobolev space on $\Omega$, and similarly use $W^{1,p}_\loc(\Omega)$ to denote the space of functions that are locally in a $W^{1,p}$-space. We recall that the Sobolev space of functions vanishing on the boundary $W^{1,p}_0(\Omega)$ is the $W^{1,p}(\Omega)$-closure of the space $C^\infty_c(\Omega)$ of smooth compactly supported functions. We also use $W^{1,p}_c(\Omega)$ to denote the space of $W^{1,p}(\Omega)$-functions with compact support in $\Omega$. By a convolution approximation argument, it is clear that $W^{1,p}_c(\Omega) \subset W^{1,n}_0(\Omega)$. Since also $C^\infty_c(\Omega) \subset W^{1,p}_c(\Omega)$, the space $W^{1,n}_0(\Omega)$ is equivalently the $W^{1,p}(\Omega)$-closure of $W^{1,p}_c(\Omega)$.
	
	We recall some basic properties of truncations of Sobolev functions. For $\phi \in W^{1,p}(\Omega)$, and $a \in \R$, if we set $u = \max(\phi, a)$, then $u \in W^{1,p}_\loc(\Omega)$, with $\nabla u(x) = \nabla \phi(x)$ for a.e.\ $x \in \phi^{-1} (-\infty, a)$ and $\nabla u(x) = 0$ for a.e.\ $x \in \phi^{-1}[a, \infty)$; see e.g.\ \cite[Theorem 1.20]{Heinonen-Kilpelainen-Martio_book}. In particular, $u \in W^{1,p}(\Omega)$ if $u \in L^p(\Omega)$, which occurs for instance if $\Omega$ is bounded or $a \le 0$. Analogous facts hold for $u = \min(\phi, a)$. 
	
	Moreover, if $\phi \in W^{1,p}_0(\Omega)$ and $\psi \in W^{1,p}(\Omega)$ are such that $\abs{\psi} \le \abs{\phi}$ a.e.\ in $\Omega$, then $\psi \in W^{1,p}_0(\Omega)$; see e.g.\ \cite[Lemma 1.25 (iii)]{Heinonen-Kilpelainen-Martio_book}. A notable consequence of these facts is that if $\phi \in W^{1,p}_0(\Omega)$, then for every $M \in [0, \infty)$, we have $(\phi - M)^+ \in W^{1,p}_0(\Omega)$. We also point out that, given $\phi \in W^{1,p}(\Omega)$ and $a \in \R$, we have $\nabla\phi(x) = 0$ for a.e.\ $x \in \phi^{-1}\{a\}$; see e.g.\ \cite[Corollary 1.21]{Heinonen-Kilpelainen-Martio_book}. 
	
	We then recall a result that allows us to select a representative of a $W^{1,p}$-function that behaves well on spheres around a fixed point.
	
	\begin{lemma}\label{lem:good_representative_lemma}
		Let $n \ge 2$, let $\Omega \subset \R^n$ be open, and let $\phi \in W^{1,p}(\Omega)$ with $p \in (n-1, \infty)$. Let $x_0 \in \Omega$, and let $R > 0$ be such that $\B^n(x_0, R) \subset \Omega$. Then there exists a representative $\tilde{\phi}$ of $\phi$ such that, for a.e.\ $r \in (0, R)$, the following properties hold:
		\begin{enumerate}
			\item \label{enum:repr_cont} $\tilde{\phi} \vert_{\S^{n-1}(x_0, r)}$ is continuous;
			\item \label{enum:repr_Morrey} for all $x, y \in \S^{n-1}(x_0, r)$, we have
			\[
				\abs{\tilde{\phi}(x) - \tilde{\phi}(y)}^p \le C(n, p) r^{p-(n-1)}
				\int_{\S^{n-1}(x_0, r)} \abs{\nabla \phi}^p \dd \cH^{n-1};
			\]
			\item \label{enum:repr_truncs} if $m, M \in \R$ are such that $m \le \tilde{\phi}(x) \le M$ for all $x \in \S^{n-1}(x_0, r)$, then
			\begin{align*}
				(\phi\vert_{\B^n(x_0, r)} -M )^{+}
				&\in W^{1,p}_0(\B^n(x_0, r))\quad\text{and}\\
				( m	- \phi\vert_{\B^n(x_0, r)} )^{+}
				&\in W^{1,p}_0(\B^n(x_0, r)).\\
			\end{align*}
		\end{enumerate}
	\end{lemma}
	\begin{proof}
		We adopt the shorthands $B(r) = \B^n(x_0, r)$ and $S(r) = \S^{n-1}(x_0, r)$ for $r > 0$.
		We select a sequence of $\phi_j \in C^\infty(\Omega)$ that converge to $\phi$ in $W^{1,p}(\Omega)$. It follows that for a.e. $r \in (0, R)$, we may pass to a subsequence $\phi_{j, r}\in C^\infty(\Omega)$ such that
		\begin{equation}\label{eq:spherical_Sobolev_conv}
			\lim_{j \to \infty} \int_{S(r)} \bigl( \abs{\phi - \phi_{j,r}}^p + \abs{\nabla \phi - \nabla \phi_{j,r}}^p \bigr) \dd \cH^{n-1}
			= 0. 
		\end{equation}
		For every $r$ such that the above holds, since $p > n-1$, we may use Morrey's inequality on spheres, see e.g.\ \cite[Lemma 2.19]{Hencl-Koskela-book}, to conclude that for all $x, y \in S(r)$ and all pairs of indices $i, j$, we have
		\begin{equation}\label{eq:Morrey_smooth}
			\abs{\phi_{j,r}(x) - \phi_{j,r}(y)}^p \le C(n, p) r^{p-(n-1)}
			\int_{S(r)} \abs{\nabla \phi_{j,r}}^p \dd \cH^{n-1}
		\end{equation}
		and
		\begin{multline}\label{eq:Morrey_smooth_differences}
			\big\lvert\abs{\phi_{i,r}(x) - \phi_{j,r}(x)} - \abs{\phi_{i,r}(y) - \phi_{j,r}(y)}\big\rvert^p\\
			\le C(n, p) r^{p-(n-1)}
			\int_{S(r)} \abs{\nabla \phi_{i,r} - \nabla \phi_{j,r}}^p \dd \cH^{n-1} 
		\end{multline}
		
		Since $\phi_{j,r}\vert_{S(r)}$ and $(\nabla \phi_{j,r})\vert_{S(r)}$ are both $L^p$-convergent, by combining \eqref{eq:Morrey_smooth_differences} with the trivial estimate
		\[
			\min_{x \in S(r)} \abs{\phi_{i,r}(x) - \phi_{j,r}(x)}^p 
			\le \frac{1}{\cH^{n-1}(S(r))} \int_{S(r)} \abs{\phi_{i,r} - \phi_{j,r}}^p \dd \cH^{n-1},
		\] 
		it follows that $\phi_{j,r}\vert_{S(r)}$ is a uniformly Cauchy sequence of continuous functions. Thus, it has a continuous uniform limit, which we select as the values of our $\tilde{\phi}\vert_{S(r)}$, ensuring that \eqref{enum:repr_cont} holds. 
		
		Since $\phi_{j,r}\vert_{S(r)}$ converges to $\phi\vert_{S(r)}$ in the $L^p$-norm, $\phi\vert_{S(r)}$ and $\tilde{\phi}\vert_{S(r)}$ differ only in a $\cH^{n-1}$-nullset. Thus, when we apply this definition of $\tilde{\phi}$ at a.e.\ $r \in (0, R)$ and set $\tilde{\phi} = \phi$ everywhere else, the resulting $\tilde{\phi}$ only differs from $\phi$ in a Lebesgue nullset. We also observe that $\tilde{\phi}$ satisfies condition \eqref{enum:repr_Morrey} by passing to the limit in \eqref{eq:Morrey_smooth}.
		
		It remains to verify condition \eqref{enum:repr_truncs}. Let $r$ be one of the radii where \eqref{eq:spherical_Sobolev_conv} holds, and suppose that $m \le \tilde{\phi}(x) \le M$ for all $x \in \S^{n-1}(x_0, r)$. We show the claim for $(\phi\vert_{B(r)} -M )^{+}$, as the proof for $( m	- \phi\vert_{B(r)} )^{+}$ is analogous. 
		
		For this, let $\eps > 0$, and observe that for large enough $j$, uniform convergence on the boundary ensures that $(\phi_{j,r}\vert_{B(r)} - (M + \eps))^+ \in W^{1,p}_c(B(r))$. 
		Since $\phi_{j,r} \to \phi$ in $W^{1,p}(\Omega)$, we may use e.g.\ \cite[Lemma 1.22]{Heinonen-Kilpelainen-Martio_book} to conclude that $(\phi_{j,r}\vert_{B(r)} - (M + \eps))^+$ converges to $(\phi\vert_{B(r)} - (M + \eps))^+$ in $W^{1,p}(B(r))$. Thus, $(\phi\vert_{B(r)} - (M + \eps))^+ \in W_0^{1,p}(B(r))$ for every $\eps > 0$. Finally, we observe that $(\phi\vert_{B(r)} - (M + \eps))^+ \to (\phi\vert_{B(r)} - M)^+$ in $W^{1,p}(B(r))$ as $\eps \to 0$, completing the proof that $(\phi\vert_{B(r)} - M)^+ \in W^{1,n}_0(B(r))$.
	\end{proof}

	\subsection{Integrals and level set methods}
	
	We define for vector fields $\xi \in L^1_\loc(\Omega, \R^n)$ that a function $\divg \xi \in L^1_\loc(\Omega)$ is a \emph{weak divergence} of $\xi$ if
	\begin{equation}\label{eq:weak_divergence}
		\int_\Omega \ip{\xi}{\nabla \eta} \dd m_n = - \int_\Omega \eta \divg \xi \dd m_n
	\end{equation}
	for all $\eta \in C^\infty_c(\Omega)$. Thus, we have $\divg \xi = 0$ weakly if and only if the left hand side of \eqref{eq:weak_divergence} vanishes for all $\eta \in C^\infty_c(\Omega)$. We then recall the following standard consequence of this definition. 
	
	\begin{lemma}\label{lem:Jacobian_zero_integral}
		Let $\Omega \subset \R^n$ be open, let $\xi \in L^{n/(n-1)}(\Omega, \R^n)$ be such that $\divg \xi = 0$ weakly, and let $\phi \in W^{1,n}_0(\Omega)$. Then
		\[
			\int_\Omega \ip{\nabla \phi}{\xi} = 0
		\]
	\end{lemma}
	\begin{proof}
		We approximate $\phi$ in the $W^{1,n}$-norm with $\phi_j \in C^\infty_c(\Omega)$, and obtain that
		\[
			\abs{\int_\Omega \ip{\nabla \phi}{\xi}}
			\le \norm{\nabla\phi-\nabla\phi_j}_{L^n(\Omega)} \norm{\xi}_{L^\frac{n}{n-1}(\Omega)}
			+ \abs{\int_\Omega \phi_j \divg \xi}
		\]
		for all $j$, where the first term tends to 0 as $j \to \infty$ and the second term vanishes since $\divg \xi = 0$.
	\end{proof}
	
	Next, we generalize Lemma \ref{lem:Jacobian_zero_integral} to the case where $\ip{\nabla \phi}{\xi}$ is weighted with a function $F(\phi)$. Analogous ideas have been used previously in the study of quasiregular values; see e.g.\ \cite[Lemma 5.3]{Kangasniemi-Onninen_Heterogeneous} or \cite[Lemma 5.2]{Kangasniemi-Onninen_Picard}. However, the specific result our proof requires has a relatively badly-behaved coefficient function $F$, which necessitates a slightly different proof than in the aforementioned articles. Note that in the following theorem, $\ip{\nabla \phi}{\xi}^{+}$ and $\ip{\nabla \phi}{\xi}^{-}$ stand for the positive and negative parts of the function $\ip{\nabla \phi}{\xi}$, respectively.
	
	\begin{lemma}\label{lem:Jacobian_zero_integral_with_F}
		Let $\Omega \subset \R^n$ be open, let $\phi \in W^{1,n}_0(\Omega)$, let $\xi \in L^{n/(n-1)}(\Omega, \R^n)$ be such that $\divg \xi = 0$ weakly, and let $F \colon [0, \infty] \to [0, \infty]$ be a non-decreasing left-continuous function. Suppose that $F(\abs{\phi(x)}) < \infty$ for a.e.\ $x \in \Omega$, and 
		\begin{equation}\label{eq:finite_integral_assumption}
			\int_\Omega F(\abs{\phi}) \ip{\nabla \phi}{\xi}^{+} < \infty
			\quad \text{or} \quad
			\int_\Omega F(\abs{\phi}) \ip{\nabla \phi}{\xi}^{-} < \infty.
		\end{equation}
		Then $F(\abs{\phi}) \ip{\nabla \phi}{\xi} \in L^1(\Omega)$, and
		\[
			\int_\Omega F(\abs{\phi}) \ip{\nabla \phi}{\xi} = 0.
		\]
	\end{lemma}
	\begin{proof}
		The claim follows by showing that 
		\begin{equation}\label{eq:J_pos_neg_part_equality}
			\int_\Omega F(\abs{\phi}) \ip{\nabla \phi}{\xi}^{+}
			= \int_\Omega F(\abs{\phi}) \ip{\nabla \phi}{\xi}^{-}.
		\end{equation}
		Indeed, since one of the integrals in \eqref{eq:J_pos_neg_part_equality} is finite by assumption, both of the integrals are then finite if \eqref{eq:J_pos_neg_part_equality} holds. Thus, it follows that $F(\abs{\phi}) \ip{\nabla \phi}{\xi} = F(\abs{\phi}) (\ip{\nabla \phi}{\xi}^{+} - \ip{\nabla \phi}{\xi}^{-})$ is integrable over $\Omega$ and has zero integral over $\Omega$.
		
		To prove \eqref{eq:J_pos_neg_part_equality}, we first let $s := \sup \{t \in [0, \infty] : F(t) < F(\infty)\}$ and fix $\eps > 0$. If $s > 0$, we use the Staircase Lemma \ref{lem:left_cont_approx} to find $ 0 = t_0 < t_1 < t_2 < \dots$ with $\lim_{k \to \infty} t_k = s$ and $\abs{F(t_{k+1}) - F(t)} \le \eps$ for all $t \in (t_{k}, t_{k+1}]$, $k \in \Z_{\ge 0}$. In the case $s = 0$, we set $t_i = 0$ for all $i$. We then define a piecewise constant staircase approximation $G \colon [0, \infty] \to [0, \infty]$ of $F$ by setting $G(0) = F(0)$, $G(t) = F(t_{k+1})$ whenever $t \in (t_k, t_{k+1}]$, $G(s) = F(s)$, and $G(t) = F(\infty)$ for $t > s$. Our construction hence ensures that $\abs{G(t) - F(t)} \le \eps$ for all $t \in [0, \infty]$.
		
		Then, for every $k \in \Z_{\ge 0}$, we define a truncation $\phi_k \colon \Omega \to \R$ of $\phi$ by
		\[
			\phi_k(x) = \begin{cases}
				t_{k+1} - t_{k}, & \phi(x) \ge t_{k+1},\\
				\phi(x) - t_k, & t_k < \phi(x) < t_{k+1},\\
				0, & -t_k \le \phi(x) \le t_k,\\
				-\phi(x) + t_k, & -t_{k+1} < \phi(x) < -t_k,\\
				-t_{k+1} + t_{k}, & \phi(x) \le -t_{k+1}\\
			\end{cases}
		\]
		for all $x \in \Omega$. It follows that $\phi_k \in W^{1,n}_0(\Omega)$, we have $\nabla \phi_k = \nabla \phi$ a.e.\ in the set $S_k := \left\{x \in \Omega : t_k < \abs{\phi(x)} < t_{k+1}\right\}$, and we moreover have $\nabla \phi_k = 0$ a.e.\ in $\Omega \setminus S_k$. Thus, 
		Lemma \ref{lem:Jacobian_zero_integral} yields that
		\begin{multline*}
			\int_{S_k} G(\abs{\phi}) \ip{\nabla \phi}{\xi}^{+}
			= \int_{\Omega} F(t_{k+1}) \ip{\nabla \phi_k}{\xi}^{+}\\
			= \int_{\Omega} F(t_{k+1}) \ip{\nabla \phi_k}{\xi}^{-}
			= \int_{S_k} G(\abs{\phi}) \ip{\nabla \phi}{\xi}^{-}.
		\end{multline*}
		
		We perform another similar truncation, and define $\phi_\infty \in W^{1,n}_0(\Omega)$ by 
		\[
			\phi_\infty(x) = \begin{cases}
				\phi(x) - s, & \phi(x) > s,\\
				0, & -s \le \phi(x) \le s,\\
				s - \phi(x), & \phi(x) < -s.
			\end{cases}
		\]
		As before, we then define $S_\infty := \left\{x \in \Omega : \abs{\phi(x)} > s\right\}$, and get
		\begin{multline*}
			\int_{S_\infty} G(\abs{\phi}) \ip{\nabla \phi}{\xi}^{+}
			= \int_{\Omega} F(\infty) \ip{\nabla \phi_\infty}{\xi}^{+}\\
			= \int_{\Omega} F(\infty) \ip{\nabla \phi_\infty}{\xi}^{-}
			= \int_{S_\infty} G(\abs{\phi}) \ip{\nabla \phi}{\xi}^{-}.
		\end{multline*}
		Note here that if $F(\infty) = \infty$, then by our assumption that $F \circ \abs{\phi} < \infty$ a.e.\ in $\Omega$, we must have that $S_\infty$ is a null-set, and both sides of the above integral are hence zero.
		
		Let $S$ then be the set of $x \in \Omega$ which are not in any of the sets $S_k, k \in \{0, \dots, \infty\}$. Then $\phi(S) \subset \{0, s, \infty, t_1, t_2, t_3, \dots \}$ is countable. Hence, we have $\nabla \phi= 0$ a.e.\ in $S$. By combining this with the above computations, we conclude that
		\[
			\int_{\Omega} G(\abs{\phi}) \ip{\nabla \phi}{\xi}^{+} 
				= \int_{\Omega} G(\abs{\phi}) \ip{\nabla \phi}{\xi}^{-}.
		\]
		However, since $\abs{G(t) - F(t)} \le \eps$ for all $t \in [0, \infty]$, the above equality yields that
		\[
			\abs{\int_{\Omega} F(\abs{\phi}) \ip{\nabla \phi}{\xi}^{+} - \int_{\Omega} F(\abs{\phi}) \ip{\nabla \phi}{\xi}^{-}}
			\le \eps \int_\Omega \abs{\ip{\nabla \phi}{\xi}}.
		\]
		Since $\phi \in W^{1,n}(\Omega)$ and $\xi \in L^{n/(n-1)}(\Omega, \R^n)$, we have $\ip{\nabla \phi}{\xi} \in L^1(\Omega)$, and letting $\eps \to 0$ hence yields the claim.
	\end{proof}
	
	\section{Distribution functions and Sobolev inequalities}\label{sect:distr_funct}
	
	In this section, we study distribution functions of measurable and Sobolev functions, culminating in a proof of Proposition \ref{prop:distr_Sobolev_ineq}.
	
	\subsection{Distribution functions} Let $\Omega \subset \R^n$ be open, and let $\phi \colon \Omega \to [0, \infty]$ be measurable. We define the \emph{upper distribution function} $\mu_\phi^{+} \colon [0, \infty] \to [0, m_n(\Omega)]$ and the \emph{lower distribution function} $\mu_\phi^{-} \colon [0, \infty] \to [0, m_n(\Omega)]$ of $\phi$ to be given by
	\begin{align}
		\label{eq:upper_Cavalieri_function}
		\mu_\phi^{+}(t) &= m_n (\{ x \in \Omega : \phi(x) \ge t \})\\
		\mu_\phi^{-}(t) &= m_n (\{ x \in \Omega : \phi(x) > t \})
	\end{align}
	for all $t \in [0, \infty]$. We also recall Cavalieri's principle, which tells us that
	\begin{equation}\label{eq:Cavalieri's_principle}
		\int_\Omega \phi \dd m_n 
		= \int_0^\infty \mu_\phi^{+}(t) \dd t
		= \int_0^\infty \mu_\phi^{-}(t) \dd t.
	\end{equation}
	
	Distribution functions are used in the study of decreasing rearrangements and symmetrizations of functions; see e.g.\ the survey monographs of Baernstein \cite{Baernstein_symmetrization} and Talenti \cite{Talenti_DecrRearr}. Such applications typically use the lower distribution function, but for our uses the upper distribution function ends up being far more useful.
	
	We point out some key properties of the two distribution functions. Both $\mu_\phi^{+}$ and $\mu_\phi^{-}$ are non-increasing, and satisfy 
	\[
	0 \le \mu_\phi^{-} \le \mu_\phi^{+} \le m_n(\Omega).
	\] 
	We observe that if $\Omega$ is bounded, then $\mu_\phi^{+}$ is left-continuous, while $\mu_\phi^{-}$ is right-continuous. For unbounded $\Omega$, the function $\mu_\phi^{-}$ remains right-continuous, but $\mu_\phi^{+}$ may have a single left jump discontinuity at $t = \sup \, (\mu_\phi^{+})^{-1}\{\infty\}$. Consequently, both $\mu_\phi^{+}$ and $\mu_\phi^{-}$ are Borel. Moreover, both $\mu_\phi^{+}$ and $\mu_\phi^{-}$ remain unchanged when $\phi$ is changed in a null-set of $\Omega$.
	
	We also note that $\mu_\phi^{+}(t) \ne \mu_\phi^{-}(t)$ only for at most countably many $t \in [0, \infty]$. Indeed, if $\mu_\phi^{+}(t) \ne \mu_\phi^{-}(t)$ for a given $t$, then $m_n(\varphi^{-1}\{t\}) > 0$. Since the sets $\varphi^{-1}\{t\}$ are pairwise disjoint and measurable, no more than countably many of them can have positive measure.
	
	We observe the following estimates for distribution functions.
	
	\begin{lemma}\label{lem:distr_level_set_bounds}
		Let $\Omega \subset \R^n$ be open and bounded, let $\phi \colon \Omega \to [0, \infty]$ be measurable, and let $a \in [0, m_n(\Omega)]$. Then we have
		\begin{align}
			\label{eq:distr_sublevel_lower_est}
			m_n \left( \{ x \in \Omega : \mu_\phi^{-}(\phi(x)) \le a\} \right) &\ge a
			\qquad \text{and}\\
			\label{eq:distr_sublevel_upper_est}
			m_n \left( \{ x \in \Omega : \mu_\phi^{+}(\phi(x)) < a\} \right) &\le a.
		\end{align}
	\end{lemma}
	\begin{proof}
		Let
		\begin{align*}
			s_a^{+} &:= \sup \{s \in [0, \infty] : \mu_\phi^{+}(s) \ge a\},\\
			s_a^{-} &:= \inf \{s \in [0, \infty] : \mu_\phi^{-}(s) \le a\}.
		\end{align*}
		Note that since $\mu_\phi^{+}(0) = m_n(\Omega)$ and $\mu_\phi^{-}(\infty) = 0$, the assumption $a \in [0, m_n(\Omega)]$ ensures that the sets used to define $s_a^{+}$ and $s_a^{-}$ are non-empty. By combining this non-emptiness property with the left-continuity of $\mu_\phi^{+}$ and the right-continuity of $\mu_\phi^{-}$, we conclude that
		\begin{equation}\label{eq:bounds_comparison_with_a}
			\mu_\phi^{-}(s_a^{-}) \le a \le \mu_\phi^{+}(s_a^{+}).
		\end{equation} 
		
		We then claim that
		\begin{equation}\label{eq:bounds_ordering}
			s_a^{-} \le s_a^{+}.
		\end{equation} 
		Indeed, suppose towards contradiction that $s_a^{-} > s_a^{+}$. In this case, $(s_a^{+}, s_a^{-})$ is a non-empty interval, and for every $s \in (s_a^{+}, s_a^{-})$, we have $\mu_\phi^{+}(s) < a < \mu_\phi^{-}(s)$ by the definitions of $s_a^{+}$ and $s_a^{-}$. Since $\mu_\phi^{+}(s) \ge \mu_\phi^{-}(s)$ for all $s$, this is a contradiction, completing the proof of \eqref{eq:bounds_ordering}.
		
		Now, we first prove \eqref{eq:distr_sublevel_lower_est}. If $x \in \Omega$ is such that $\phi(x) \ge s_a^{-}$, then since $\mu_\phi^{-}$ is non-increasing, \eqref{eq:bounds_comparison_with_a} yields that $\mu_\phi^{-}(\phi(x)) \le \mu_\phi^{-}(s_a^{-}) \le a$. Thus, we conclude that
		\[
		\{ x \in \Omega : \phi(x) \ge s_a^{-}\} \subset 
		\{ x \in \Omega : \mu_\phi^{-}(\phi(x)) \le a \}.
		\]
		It hence follows using \eqref{eq:bounds_comparison_with_a}, \eqref{eq:bounds_ordering}, and the fact that $\mu_\phi^{+}$ is non-increasing, that
		\[
		m_n\left( \{ x \in \Omega : \mu_\phi^{-}(\phi(x)) \le a \} \right) 
		\ge \mu_\phi^{+}(s_a^{-})
		\ge \mu_\phi^{+}(s_a^{+})
		\ge a,
		\]
		completing the proof of \eqref{eq:distr_sublevel_lower_est}.
		
		It remains to prove \eqref{eq:distr_sublevel_upper_est}. We first note that if $x \in \Omega$ is such that $\phi(x) \le s_a^{+}$, then \eqref{eq:bounds_comparison_with_a} and the fact that $\mu_{\phi}^{+}$ is non-increasing yield that $\mu_\phi^{+}(\phi(x)) \ge \mu_\phi^{+}(s_a^{+}) \ge a$. Thus, $\left\{ x \in \Omega : \phi(x) \le s_a^{+}\right\} \subset \left\{ x \in \Omega : \mu_\phi^{+}(\phi(x)) \ge a \right\}$, which implies that
		\[
		\{ x \in \Omega : \mu_\phi^{+}(\phi(x)) < a \}
		\subset \{ x \in \Omega : \phi(x) > s_a^{+}\}.
		\]
		Now, by again using \eqref{eq:bounds_comparison_with_a}, \eqref{eq:bounds_ordering}, and the non-increasing property of $\mu_\phi^{-}$, it follows that
		\[
		m_n\left( \{ x \in \Omega : \mu_\phi^{+}(\phi(x)) < a \} \right) 
		\le \mu_\phi^{-}(s_a^{+})
		\le \mu_\phi^{-}(s_a^{-})
		\le a,
		\]
		completing the proof of \eqref{eq:distr_sublevel_upper_est}.
	\end{proof}
	
	We then study integrals of powers of distribution functions composed with the original function. We begin with the result for negative exponents.
	
	\begin{lemma}\label{lem:inv_cav_funct_integral}
		Let $\Omega \subset \R^n$ be open and bounded, and let $\phi \colon \Omega \to [0, \infty]$ be measurable. Then for every $\gamma \in (0, 1)$, we have
		\[
		\int_\Omega \frac{1}{(\mu_\phi^{+} \circ \phi)^\gamma} \dd m_n 
		\le \frac{1}{1 - \gamma} [m_n(\Omega)]^{1-\gamma}
		\le \int_\Omega \frac{1}{(\mu_\phi^{-} \circ \phi)^\gamma} \dd m_n.
		\]
		Moreover, for every $\gamma \in [1, \infty)$, we have
		\[
		\int_\Omega \frac{1}{(\mu_\phi^{-} \circ \phi)^\gamma} \dd m_n
		= \infty.
		\]
	\end{lemma}
	\begin{proof}
		Note that the integrands are measurable since $\phi$ is measurable and $\mu_\phi^{\pm}$ are Borel. We first prove the estimate for $\mu_\phi^{+}$. We apply Cavalieri's principle, notably the lower distribution function version, obtaining that
		\begin{align*}
			\int_\Omega \frac{1}{(\mu_\phi^{+} \circ \phi)^\gamma} \dd m_n
			&= \int_0^\infty m_n \left(\left\{ x \in \Omega : \frac{1}{(\mu_\phi^{+}(\phi(x)))^\gamma} > t \right\}\right) \dd t\\
			&= \int_0^\infty m_n \left(\left\{ x \in \Omega : \mu_\phi^{+}(\phi(x)) < \frac{1}{t^{1/\gamma}} \right\}\right) \dd t.
		\end{align*}
		Note that for $t > 0$, if $t^{-1/\gamma} > m_n(\Omega)$, i.e.\ if $t < [m_n(\Omega)]^{-\gamma}$, then the set in the integral is all of $\Omega$. On the other hand, if $t^{-1/\gamma} \le m_n(\Omega)$, then we may apply the upper bound of Lemma \ref{lem:distr_level_set_bounds} \eqref{eq:distr_sublevel_upper_est}. Thus, we conclude that
		\begin{align*}
			\int_\Omega \frac{1}{(\mu_\phi^{+} \circ \phi)^\gamma} \dd m_n
			&\leq \int_0^{[m_n(\Omega)]^{-\gamma}} m_n(\Omega) \dd t + \int_{[m_n(\Omega)]^{-\gamma}}^\infty \frac{1}{t^{1/\gamma}} \dd t\\
			&= [m_n(\Omega)]^{1-\gamma} + \frac{\gamma}{1 - \gamma}[m_n(\Omega)]^{1-\gamma}
			= \frac{1}{1 - \gamma}[m_n(\Omega)]^{1-\gamma}
		\end{align*}
		when $0 < \gamma < 1$.
		
		The proof for $\mu_\phi^{-}$ is similar. Indeed, with the upper distribution function version of Cavalieri's principle, we get that
		\[
		\int_\Omega \frac{1}{(\mu_\phi^{-} \circ \phi)^\gamma} \dd m_n
		= \int_0^\infty m_n \left(\left\{ x \in \Omega : \mu_\phi^{+}(\phi(x)) \le \frac{1}{t^{1/\gamma}} \right\}\right) \dd t.
		\]
		We then again split to two regions as above, except this time we use the lower bound of Lemma \ref{lem:distr_level_set_bounds} \eqref{eq:distr_sublevel_lower_est} in the appropriate region. It follows that
		\[
		\int_\Omega \frac{1}{(\mu_\phi^{-} \circ \phi)^\gamma} \dd m_n
		\ge [m_n(\Omega)]^{1-\gamma} + \int_{[m_n(\Omega)]^{-\gamma}}^\infty \frac{1}{t^{1/\gamma}} \dd t,
		\]
		where this lower bound equals $(1-\gamma)^{-1} [m_n(\Omega)]^{1-\gamma}$ when $0 < \gamma < 1$ and diverges to $\infty$ when $\gamma \ge 1$.
	\end{proof}
	
	For completeness, we also provide an analogous integral estimate for positive exponents; note that the order of the bounds ends up reversed.
	
	\begin{lemma}\label{lem:cav_funct_integral}
		Let $\Omega \subset \R^n$ be open and bounded, and let $\phi \colon \Omega \to [0, \infty]$ be measurable. Then for every $r \in (0, \infty)$, we have
		\[
		\int_\Omega (\mu_\phi^{-} \circ \phi)^r \dd m_n 
		\le \frac{1}{1+r}[m_n(\Omega)]^{1+r}
		\le \int_\Omega (\mu_\phi^{+} \circ \phi)^r \dd m_n.
		\]
	\end{lemma}
	\begin{proof}
		We first prove the lower bound for $(\mu_\phi^{+} \circ \phi)^r$. We start similarly as in the proof of Lemma \ref{lem:inv_cav_funct_integral}, and apply Cavalieri's principle to conclude that
		\begin{align*}
			\int_\Omega (\mu_\phi^{+} \circ \phi)^r \dd m_n
			&= \int_0^\infty m_n \left(\left\{ x \in \Omega : 
			(\mu_\phi^{+}(\phi(x)))^r \ge t \right\}\right) \dd t\\
			&= \int_0^{\infty} m_n \left(\left\{ x \in \Omega : \mu_\phi^{+}(\phi(x)) \ge t^{1/r} \right\}\right) \dd t,
		\end{align*}
		Notably, if $t > [m_n(\Omega)]^r$, then we have $\mu^{+}_\phi(\phi(x)) \le m_n(\Omega) < t^{1/r}$ for all $x \in \Omega$, implying that the set in the above integral is empty for such $t$. Thus,
		\[
		\int_\Omega (\mu_\phi^{+} \circ \phi)^r \dd m_n
		= \int_0^{[m_n(\Omega)]^r} m_n \left(\left\{ x \in \Omega : \mu_\phi^{+}(\phi(x)) \ge t^{1/r} \right\}\right) \dd t.
		\]
		Next, note that $\left\{ x \in \Omega : \mu_\phi^{+}(\phi(x)) \ge t^{1/r} \right\} = \Omega \setminus \left\{ x \in \Omega : \mu_\phi^{+}(\phi(x)) < t^{1/r} \right\}$. Thus, the above integral becomes
		\begin{multline*}
			\int_\Omega (\mu_\phi^{+} \circ \phi)^r \dd m_n\\
			= [m_n(\Omega)]^{r+1} - \int_0^{[m_n(\Omega)]^r} m_n \left(\left\{ x \in \Omega : \mu_\phi^{+}(\phi(x)) < t^{1/r} \right\}\right) \dd t.
		\end{multline*}
		Now, the final step is to use the upper bound of Lemma \ref{lem:distr_level_set_bounds} \eqref{eq:distr_sublevel_upper_est}, and compute that
		\begin{align*}
			\int_\Omega (\mu_\phi^{+} \circ \phi)^r \dd m_n
			&\ge [m_n(\Omega)]^{r+1} - \int_0^{[m_n(\Omega)]^r} t^{1/r} \dd t
			= \frac{1}{1+r}[m_n(\Omega)]^{1+r},
		\end{align*}
		proving the desired estimate for $(\mu_\phi^{+} \circ \phi)^r$.
		
		The upper bound for $(\mu_\phi^{-} \circ \phi)^r$ is proven analogously. Indeed, if we instead start with the lower distribution function version of Cavalieri's principle, the above computation, mutatis mutandis, yields that
		\begin{multline*}
			\int_\Omega (\mu_\phi^{-} \circ \phi)^r \dd m_n\\
			= [m_n(\Omega)]^{r+1} - \int_0^{[m_n(\Omega)]^r} m_n \left(\left\{ x \in \Omega : \mu_\phi^{-}(\phi(x)) \le t^{1/r} \right\}\right) \dd t.
		\end{multline*}
		The lower bound of Lemma \ref{lem:distr_level_set_bounds} \eqref{eq:distr_sublevel_lower_est} then yields the desired
		\[
		\int_\Omega (\mu_\phi^{-} \circ \phi)^r \dd m_n
		\le [m_n(\Omega)]^{r+1} - \int_0^{[m_n(\Omega)]^r} t^{1/r} \dd t
		= \frac{1}{1+r}[m_n(\Omega)]^{1+r}.
		\]
	\end{proof}
	
	\subsection{Proof of Proposition \ref{prop:distr_Sobolev_ineq}}
	
	We then proceed to prove Proposition \ref{prop:distr_Sobolev_ineq}. Note that if we write the statement using the upper distribution function, the inequality reads as
	\[
	\norm{\phi}_{L^\infty(\Omega)} \le \frac{1}{n \omega_n^{1/n}} 
	\int_\Omega 
	\frac{\abs{\nabla \phi}}{[\mu_\phi^{+} \circ \phi]^{\frac{n-1}{n}}}
	\dd m_n.
	\]
	Proposition \ref{prop:distr_Sobolev_ineq} is closely tied to the isoperimetric inequality; indeed, for $\phi \in C^\infty_c(\Omega)$, one can easily derive it using the coarea formula and the classical isoperimetric inequality. Here, we choose an approach which uses the sharp Sobolev inequality for $W^{1,1}$-functions, which is famously known to be equivalent to the isoperimetric inequality. We first recall this sharp Sobolev inequality; see e.g.\ \cite[Section 4.6]{Baernstein_symmetrization} for details.
	
	\begin{lemma}\label{lem:sharp_1-Sobolev_ineq}
		Let $n \ge 2$, let $\Omega \subset \R^n$ be open, and let $\phi \in W^{1,1}_0(\Omega)$. Then
		\[
		\left( \int_\Omega \abs{\phi}^\frac{n}{n-1} \dd m_n \right)^\frac{n-1}{n}
		\le  \frac{1}{n \omega_n^{1/n}} \int_{\Omega} \abs{\nabla\phi} \dd m_n.
		\]
	\end{lemma}
	
	The key lemma we use to show Proposition \ref{prop:distr_Sobolev_ineq} is as follows.
	
	\begin{lemma}\label{lem:bound_difference_by_gradient}
		Let $n \ge 2$, let $\Omega \subset \R^n$ be open, and let $\phi \in W^{1,1}_0(\Omega)$ with $\phi \ge 0$ a.e.\ in $\Omega$. Then for all $0 \le a < b < \norm{\phi}_{L^\infty(\Omega)}$, we have
		\[
		(b-a) \le \frac{1}{n \omega_n^{1/n}} \int_{\phi^{-1}(a, b)} 
		\frac{\abs{\nabla\phi}}{[\mu^{+}_\phi(b)]^\frac{n-1}{n}} \dd m_n.
		\]
	\end{lemma}
	\begin{proof}
		Let $0 \le a < b < \norm{\phi}_{L^\infty(\Omega)}$, and define a truncated function $\psi \colon \Omega \to \R$ by
		\[
		\psi(x) := \begin{cases}
			0 & \phi(x) \le a,\\
			\phi(x) - a & a < \phi(x) < b,\\
			b - a & \phi(x) \ge b
		\end{cases}
		\]
		for all $x \in \Omega$. Since $\abs{\psi} \le \phi$, it follows that $\psi$ is in $W^{1,1}_0(\Omega)$. Moreover, we have $\nabla \psi(x) = \nabla \phi(x)$ for a.e.\ $x \in \Omega$ with $a < \psi(x) < b$, and $\nabla \psi(x) = 0$ for a.e.\ $x \in \Omega$ with $\psi(x) \ge b$ or $\psi(x) \le a$. 
		
		We then apply Lemma \ref{lem:sharp_1-Sobolev_ineq} on $\psi$. It follows that
		\begin{multline*}
			\frac{1}{n \omega_n^{1/n}} \int_{\phi^{-1}(a,b)} \abs{\nabla\phi} \dd m_n
			= \frac{1}{n \omega_n^{1/n}} \int_{\Omega} \abs{\nabla\psi} \dd m_n\\
			\ge \left( \int_\Omega \abs{\psi}^\frac{n}{n-1} \dd m_n \right)^\frac{n-1}{n}
			\ge \left( \int_{\phi^{-1}[b, \infty]} \abs{b-a}^\frac{n}{n-1} 
			\dd m_n \right)^\frac{n-1}{n}\\ 
			= (b-a) [\mu^{+}_\phi(b)]^\frac{n-1}{n}.
		\end{multline*}
		We then divide on both sides by $[\mu^{+}_\phi(b)]^\frac{n-1}{n}$, which is non-zero since $b < \norm{\phi}_{L^\infty(\Omega)}$ and finite since $\varphi \in L^1(\Omega)$. The claim hence follows.
	\end{proof}
	
	We then note that if $\Omega \subset \R^n$ is open and if $\phi \colon \Omega \to [0, \infty]$ is measurable, then $\mu_\phi^{+}(t) = 0$ when $t > \norm{\phi}_{L^\infty(\Omega)}$, and $\mu_\phi^{+}(t) > 0$ when $t < \norm{\phi}_{L^\infty(\Omega)}$. Hence, if $\mu_\phi^{+}(t) < \infty$ for all $t > 0$, then $\mu_\phi^{+}$ is left-continuous, and we may apply the Staircase Lemma \ref{lem:left_cont_approx} on $1/(\mu_\phi^{+})^\gamma$ with $s = \norm{\phi}_{L^\infty(\Omega)}$. Thus, we obtain the following approximation result. 
	
	\begin{cor}\label{cor:distr_inv_approx}
		Let $\Omega \subset \R^n$ be open, let $\phi \colon \Omega \to [0, \infty]$ be measurable, and let $\gamma > 0$. Suppose that $\mu_\phi^{+}(t) < \infty$ for all $t > 0$. If $\norm{\phi}_{L^\infty(\Omega)} > 0$, then for every $\eps > 0$, there exists an increasing sequence $t_0 < t_1 < t_2 < \dots$ for which $t_0 = 0$, $\lim_{i \to \infty} t_i = \norm{\phi}_{L^\infty(\Omega)}$, and for every $i \in \Z_{>0}$ and all $t \in (t_{i-1}, t_i]$, we have
		\[
		\abs{\frac{1}{[\mu_\phi^{+}(t_i)]^\gamma} - \frac{1}{[\mu_\phi^{+}(t)]^\gamma}} \le \eps.
		\]
	\end{cor}
	
	Proposition \ref{prop:distr_Sobolev_ineq} now follows by combining Lemma \ref{lem:bound_difference_by_gradient} with Corollary \ref{cor:distr_inv_approx}.
	
	\begin{proof}[Proof of Proposition \ref{prop:distr_Sobolev_ineq}]
		Fix $\eps > 0$. If $\norm{\phi}_{L^\infty(\Omega)} = 0$, then the claim is trivial. Otherwise, since $\phi \in L^1(\Omega)$, we have $\mu_\phi^{+}(t) < \infty$ for all $t > 0$, and we may hence use Corollary \ref{cor:distr_inv_approx} to select an increasing sequence $t_0 < t_1 < t_2 < \dots$ in $[0, \norm{\phi}_{L^\infty(\Omega)})$ for which $t_0 = 0$, $\lim_{i \to \infty} t_i = \norm{\phi}_{L^\infty(\Omega)}$, and
		\begin{equation}\label{eq:upper_bound_eps_est}
			\frac{1}{[\mu^{+}_\phi(t_{i+1})]^{\frac{n-1}{n}}}  
			\le \frac{1}{[\mu^{+}_\phi(t)]^{\frac{n-1}{n}}} + \eps
		\end{equation}
		for all $t \in (t_{i}, t_{i+1}]$, $i \in \Z_{\ge 0}$. Now, by Lemma \ref{lem:bound_difference_by_gradient}, we have
		\[
		\norm{\phi}_{L^\infty(\Omega)}
		= \sum_{i = 0}^\infty (t_{i+1} - t_i)
		\le \frac{1}{n \omega_n^{1/n}} \sum_{i = 0}^\infty \int_{\phi^{-1}(t_i, t_{i+1})} 
		\frac{\abs{\nabla\phi(x)}}{[\mu^{+}_\phi(t_{i+1})]^\frac{n-1}{n}} \dd m_n(x).
		\]
		Moreover, \eqref{eq:upper_bound_eps_est} ensures that for every $x \in \phi^{-1}(t_i, t_{i+1})$, we have
		\[
		\frac{1}{[\mu^{+}_\phi(t_{i+1})]^\frac{n-1}{n}}
		\le \frac{1}{[\mu^{+}_\phi(\phi(x))]^\frac{n-1}{n}} + \eps.
		\]
		Since the sets $\phi^{-1}(t_i, t_{i+1})$ are also disjoint, we may hence combine the integrals, obtaining that
		\[
		\norm{\phi}_{L^\infty(\Omega)}
		\le \frac{1}{n \omega_n^{1/n}} \left( 
		\int_\Omega \frac{\abs{\nabla\phi}}{[\mu^{+}_\phi \circ \phi]^\frac{n-1}{n}} \dd m_n
		+ \eps \int_\Omega \abs{\nabla\phi} \dd m_n
		\right).
		\]
		Now, since $\abs{\nabla \phi} \in L^1(\Omega)$, the claim follows by letting $\eps \to 0$ in the above upper bound.
	\end{proof}
	
	\section{Proof of Proposition \ref{prop:weak_weak_mono_cont}}\label{sect:almost_weak_mono_is_cont}
	
	Recall the definition of $\alpha$-almost weakly monotone $W^{1,1}_\loc$-functions from Definition \ref{def:weak_weak_monotone}. We begin by pointing out that if $\phi \in W^{1,p}_\loc(\Omega)$, then all instances of $W^{1,1}_0$-spaces in the definition can be replaced with $W^{1,p}_0$-spaces.
	
	\begin{lemma}\label{lem:almost_weak_mono_exp_p}
		Let $n \ge 2$, let $\Omega \subset \R^n$ be open, let $p \in [1, \infty)$, and let $\phi \in W^{1,p}_\loc(\Omega)$. Then $\phi$ is $\alpha$-almost weakly monotone with constants $C, \alpha > 0$ if and only if, whenever $B_r \subset \Omega$ is a ball of radius $r$ that is compactly contained in $\Omega$, and $M, m \in \R$ are such that
		\[
			(\varphi \vert_{B_r} - M)^+ \in W^{1,p}_0(B_r)
			\quad \text{and} \quad
			(m - \varphi \vert_{B_r})^+ \in W^{1,p}_0(B_r),
		\]
		we have
		\[
			\norm{(\varphi \vert_{B_r} - M)^+}_{L^\infty(B_r)} \le Cr^{\alpha}
			\quad \text{and} \quad
			\norm{(m - \varphi \vert_{B_r})^+}_{L^\infty(B_r)} \le Cr^{\alpha}.
		\]
	\end{lemma}
	\begin{proof}
		The claim follows immediately from the fact that for all balls $B_r$ as above, we have
		\begin{equation}\label{eq:W1p0_space_characterization}
			W^{1,p}_0(B_r) = W^{1,p}(B_r) \cap W^{1,1}_0(B_r).
		\end{equation}
		For convenience, we recall a proof of the characterization \eqref{eq:W1p0_space_characterization}. The inclusion $W^{1,p}_0(B_r) \subset W^{1,p}(B_r) \cap W^{1,1}_0(B_r)$ is trivial. For the other inclusion in the non-trivial case $p > 1$, if $\phi \in W^{1,p}(B_r) \cap W^{1,1}_0(B_r)$, one observes that the zero extension $\phi_0 \in L^p(\R^n)$ of $\phi$ is first in $W^{1,1}(\R^n)$ since $\phi \in W^{1,1}_0(B_r)$, and then consequently in $W^{1,p}(\R^n)$ since $\nabla \phi_0 = \nabla \phi$ a.e.\ in $\Omega$ and $\nabla \phi_0 = 0$ a.e.\ in $\R^n \setminus \Omega$. Thus, since $\partial B_r$ is $C^1$-smooth, we have $\phi \in W^{1,p}_0(B_r)$; see e.g.\ \cite[Proposition 9.18]{Brezis_book}.
	\end{proof}

	Our objective in this section is to prove Proposition \ref{prop:weak_weak_mono_cont}, which states that almost weakly monotone $W^{1,n}_\loc$-functions have a continuous representative. For this, given a measurable function $\phi \colon \Omega \to \R$ and a measurable set $A \subset \Omega$, we define that the \emph{oscillation} of $\phi$ over $A$ is given by
	\[
		\osc_A(\phi) := \sup_{x \in A} \phi(x) - \inf_{y \in A} \phi(y),
	\]
	and the \emph{essential oscillation} of $\phi$ over $A$ is given by
	\[
		\essosc_A(\phi) := \esssup_{x \in A} \phi(x) - \essinf_{y \in A} \phi(y).
	\]
	Our strategy, which parallels the standard approach for proving the continuity of weakly monotone functions, is to use the following lemma.
	\begin{lemma}\label{lem:essosc_implies_cont}
		Let $\Omega \subset \R^n$ be open, and let $\phi \colon \Omega \to \R$ be measurable. Suppose that for every $x_0 \in \Omega$, we have
		\[
			\lim_{r \to 0} \essosc_{\B^n(x_0, r)}(\phi) = 0.
		\]
		Then there is a continuous function $\tilde{\phi} \colon \Omega \to \R$ such that $\tilde{\phi}(x) = \phi(x)$ for a.e.\ $x \in \Omega$.
	\end{lemma}
	\begin{proof}
		Let $x_0 \in \Omega$. We define
		\begin{align*}
			\esslimsup_{x \to x_0} \phi(x) &:=
			\lim_{r \to 0} \esssup_{x \in \B^n(x_0, r)} \phi(x)
			\quad \text{and}\\
			\essliminf_{x \to x_0} \phi(x) &:=
			\lim_{r \to 0} \essinf_{x \in \B^n(x_0, r)} \phi(x)
		\end{align*}
		Both of these limits are guaranteed to exist, since $r \mapsto \esssup_{x \in \B^n(x_0, r)} \phi(x)$ is non-decreasing and $r \mapsto \essinf_{x \in \B^n(x_0, r)} \phi(x)$ is non-increasing. It is also clear that $\essliminf_{x \to x_0} \phi(x) \le \esslimsup_{x \to x_0} \phi(x)$. Moreover, our assumption ensures that $\esslimsup_{x \to x_0} \phi(x) = \essliminf_{x \to x_0} \phi(x)$. Thus, we define
		\[
			\tilde{\phi}(x_0) := \esslimsup_{x \to x_0} \phi(x) = \essliminf_{x \to x_0} \phi(x).
		\]
		
		We observe that whenever $0 < r < \dist(x_0, \partial \Omega)$, we have
		\[
			\essinf_{x \in \B^n(x_0, r)} \phi(x)
			\le \phi_{\B^n(x_0, r)}
			\le \esssup_{x \in \B^n(x_0, r)} \phi(x).
		\]
		Thus, if $x_0$ is a Lebesgue point of $\phi$, then $\phi(x) = \lim_{r \to 0} \phi_{\B^n(x_0, r)} = \tilde{\phi}(x)$, showing that $\phi$ and $\tilde{\phi}$ indeed agree a.e.\ in $\Omega$.
		
		It remains to verify that $\tilde{\phi}$ is continuous. Thus, let $\eps > 0$, let $\delta > 0$ be such that $\B^n(x_0, \delta) \subset \Omega$ and $\essosc_{\B^n(x_0, \delta)}(\phi) < \eps$, and suppose that $y_0 \in \B^n(x_0, \delta)$. Then
		\begin{multline*}
			\tilde{\phi}(x_0) \le \esssup_{x \in \B^n(x_0, \delta)} \phi(x)
			\le \eps + \essinf_{x \in \B^n(x_0, \delta)} \phi(x)\\
			\le \eps + \essinf_{x \in \B^n(y_0, \delta - \abs{x_0 - y_0})} \phi(x)
			\le \eps + \tilde{\phi}(y_0).
		\end{multline*}
		In an analogous manner with the suprema and infima reversed, we get that $\tilde{\phi}(x_0) \ge \tilde{\phi}(y_0) - \eps$. Thus $\abs{\tilde{\phi}(x_0) - \tilde{\phi}(y_0)} < \eps$ for all $y_0 \in \B^n(x_0, \delta)$, proving the continuity of $\tilde{\phi}$.
	\end{proof}
	
	Thus, Proposition \ref{prop:weak_weak_mono_cont} follows by proving the following lemma.
	
	\begin{lemma}\label{lem:ocs_est_for_weak_pseudomonotone}
		Let $n \ge 2$, let $\Omega \subset \R^n$ be open, and let $\phi \in W^{1,n}_\loc(\Omega)$. If $\phi$ is $\alpha$-almost weakly monotone for some $\alpha > 0$, then for every $x_0 \in \Omega$, we have
		\[
			\lim_{r \to 0} \essosc_{\B^n(x_0, r)}(\phi) = 0,
		\]
		and more precisely, $\essosc_{\B^n(x_0, r)}(\phi) = O(\log^{-1/n}(1/r))$ as $r \to 0$.
	\end{lemma}
	\begin{proof}
		Fix a point $x_0 \in \Omega$, and for all $r > 0$, denote $B(r) = \B^n(x_0, r)$ and $S(r) = \S^{n-1}(x_0, r)$. Fix a radius $R > 0$ for which $B(R)$ is compactly contained in $\Omega$. With this $x_0$ and $R$, we then use Lemma \ref{lem:good_representative_lemma} to select a representative $\tilde{\phi}$ of $\phi$.
		
		We fix a $r \in (0, R)$. For every $\rho\in (0, R)$, we let 
		\[
			M(\rho) := \sup_{x \in S(\rho)} \tilde{\phi}(x)
			\quad\text{and}\quad
			m(\rho) := \inf_{x \in S(\rho)} \tilde{\phi}(x).
		\] 
		It follows by parts \eqref{enum:repr_cont} and \eqref{enum:repr_truncs} of Lemma \ref{lem:good_representative_lemma} that for a.e.\ $\rho\in (0, R)$, $m(\rho)$ and $M(\rho)$ are finite,
		\begin{align*}
			(\phi\vert_{B(\rho)} - M(\rho))^+ &\in W^{1,n}_0(B(\rho)),
			\quad \text{and}\\
			(m(\rho) - \phi\vert_{B(\rho)})^+ &\in W^{1,n}_0(B(\rho)).
		\end{align*}
		The definition of almost weak monotonicity then yields for a.e.\ $\rho\in (0, R)$ that
		\begin{align*}
			\esssup_{x \in B(\rho)} \phi(x) &\le M(\rho) + C\rho^{\alpha}
			\quad\text{and}\\
			\essinf_{x \in B(\rho)} \phi(x) &\ge m(\rho) - C\rho^{\alpha}.
		\end{align*}
		Thus, we obtain the bound
		\begin{equation}\label{eq:ess_osc_bound}
			\Bigl[\essosc_{\B^n(x_0, \rho)}(\phi)\Bigr]^n
			\le 2^{n-1} (M(\rho) - m(\rho))^{n} + 2^n C^n \rho^{n\alpha}
		\end{equation}
		for a.e.\ $\rho\in (0, R)$.
		
		Now, by part \eqref{enum:repr_cont} of Lemma \ref{lem:good_representative_lemma}, we observe that $M(\rho)$ and $m(\rho)$ are in $\tilde{\phi} (S(\rho))$ for a.e.\ $\rho\in (0, R)$. For such $\rho$, part \eqref{enum:repr_Morrey} of Lemma \ref{lem:good_representative_lemma} yields that
		\[
			(M(\rho) - m(\rho))^n \le C(n) \rho\int_{S(\rho)} \abs{\nabla \phi}^n \dd \cH^{n-1}
		\]
		We combine this with \eqref{eq:ess_osc_bound}, divide on both sides by $\rho$, and integrate with respect to $\rho$ from $r$ to $R$. The result is that
		\begin{multline*}
			\int^R_{r} \frac{[\essosc_{\B^n(x_0, \rho)}(\phi)]^n}{\rho} \dd \rho\\
			\le C(n) \int^R_{r} \left(C^n \rho^{n\alpha - 1} + \int_{S(\rho)} \abs{\nabla \phi}^n \dd \cH^{n-1}\right) \dd \rho\\ 
			\leq C(n) \left(\frac{C^n}{n\alpha} R^{n\alpha} + \norm{\nabla \phi}_{L^n(B(R))}^n \right).
		\end{multline*}
		This upper bound is finite, since $\phi \in W^{1,n}(B(R))$ and $\alpha > 0$.
		
		However, we may estimate in the other direction that
		\[
			\int^R_{r} \frac{[\essosc_{\B^n(x_0, \rho)}(\phi)]^n}{\rho} \dd \rho
			\ge \left[\essosc_{\B^n(x_0, r)}(\phi)\right]^n \int^R_r \frac{\dd \rho}{\rho}
			= \left[\essosc_{\B^n(x_0, r)}(\phi)\right]^n \log \frac{R}{r}.
		\]
		Thus, we get that
		\[
			\essosc_{\B^n(x_0, r)}(\phi) \le C(n) \left(\frac{C^n}{n\alpha} R^{n\alpha} + \norm{\nabla \phi}_{L^n(B(R))}^n \right)^\frac{1}{n} \log^{-\frac{1}{n}} \biggl(\frac{R}{r}\biggr), 
		\]
		which proves that $\essosc_{\B^n(x_0, r)}(\phi) = O(\log^{-1/n}(R/r)) = O(\log^{-1/n}(1/r))$ as $r \to 0$, and in particular, that $\lim_{r \to 0} \essosc_{\B^n(x_0, r)}(\phi) = 0$. 
	\end{proof}
	
	Thus, with Lemmas \ref{lem:essosc_implies_cont} and \ref{lem:ocs_est_for_weak_pseudomonotone}, the proof of Proposition \ref{prop:weak_weak_mono_cont} is complete.
	
	\section{Proof of Theorem \ref{thm:continuity_general}}\label{sect:main_result}
	
	We then prove our main results. We begin with Theorem \ref{thm:weakly_vanishing_L_infinity_estimate}; we first recall the statement for the convenience of the reader.
	
	\Linftyestimate*

	\begin{proof}
		We first note that we may assume that $K(x) > 0$ for all $x \in \Omega$. Indeed, if $x \in \Omega$ is such that $K(x) = 0$ and \eqref{eq:divergence_distortion_ineq} applies, then we have $\abs{\nabla\phi(x)}^n \le 0$, and hence $\nabla \phi(x) = 0$. Thus, if we change the value of $K(x)$ to $K(x) = 1$ at this point, then \eqref{eq:divergence_distortion_ineq} remains valid since $\abs{\nabla \phi(x)}^n = 0 \le A(x) = \ip{\nabla \phi(x)}{\xi(x)} + A(x)$. Moreover, changing the values of $K$ to $1$ in $K^{-1}\{0\}$ retains the fact that $K$ is measurable and in $L^p_\loc(\Omega)$.
		
		We may also, for convenience, assume that $\phi \ge 0$ in all of $\Omega$. This is due to the observation that
		\begin{align*}
			\abs{\nabla \phi^{+}(x)}^n &\le K(x)\left( \ip{\nabla \phi^{+}(x)}{\xi(x)} + A(x) \right) 
				\qquad \text{and}\\
			\abs{\nabla \phi^{-}(x)}^n &\le K(x)\left( \ip{\nabla \phi^{-}(x)}{-\xi(x)} + A(x) \right)
		\end{align*}
		for a.e.\ $x \in \Omega$, since at every $x$, each of the above inequalities is either the assumed \eqref{eq:divergence_distortion_ineq} for $\varphi$ or the trivial estimate $0 \le K(x) A(x)$. Moreover, proving the claim for $\phi^{+}$ and $\phi^{-}$ clearly implies it for $\phi$.
		
		Next, we observe that since $K > 0$, \eqref{eq:divergence_distortion_ineq} implies that $0 \le \abs{\nabla \phi(x)}^n/K(x) \le \ip{\nabla \phi(x)}{\xi(x)} + A(x)$ for a.e.\ $x \in \Omega$. Consequently, we have
		\begin{equation}\label{eq:Jacobian_neg_part_bound}
			\ip{\nabla \phi(x)}{\xi(x)}^{-} \le A(x)
		\end{equation}
		for a.e.\ $x \in \Omega$. 
		
		We then fix a $\gamma \in (p^{-1}, 1 - q^{-1})$; since $p^{-1} + q^{-1} < 1$, this interval is non-empty, and the selection can be made in a way that depends only on $p$ and $q$. We observe that $\mu_\phi^{+} \circ \phi$ is finite and positive for a.e.\ $x \in \Omega$; indeed, we only have $\mu_\phi^{+}(\phi(x)) = 0$ if $\phi(x) > \norm{\phi}_{L^{\infty}(\Omega)}$, or if $\phi(x) = \norm{\phi}_{L^{\infty}(\Omega)}$ and $\phi^{-1}\{\norm{\phi}_{L^{\infty}(\Omega)}\}$ is a null-set. We then estimate using \eqref{eq:Jacobian_neg_part_bound} and H\"older's inequality that
		\begin{multline}\label{eq:Jacobian_neg_part_int_bound}
			\int_\Omega \frac{\ip{\nabla \phi(x)}{\xi(x)}^{-}}{(\mu_\phi^{+} \circ \phi)^{\gamma}}
			\le \int_\Omega \frac{A}{(\mu_\phi^{+} \circ \phi)^{\gamma}}\\
			\le \norm{A}_{L^q(\Omega)} \left( \int_\Omega \frac{1}{(\mu_\phi^{+} \circ \phi)
				^{\frac{\gamma q}{q-1}}} \right)^{1 - \frac{1}{q}}.
		\end{multline}
		
		Since $0 < \gamma < 1 - q^{-1}$, we have $0 < \gamma q/(q-1) < 1$. Thus, Lemma \ref{lem:inv_cav_funct_integral} shows that the upper bound of \eqref{eq:Jacobian_neg_part_int_bound} is finite. Since also $\phi \in W^{1,n}_0(\Omega)$ and $\phi \ge 0$ by assumption, Lemma \ref{lem:Jacobian_zero_integral_with_F} now applies with $F = (\mu_\phi^{+})^{-\gamma}$, and yields that
		\[
			\int_\Omega \frac{\ip{\nabla \phi(x)}{\xi(x)}}{(\mu_\phi^{+} \circ \phi)^{\gamma}} = 0.
		\]
		By combining this with \eqref{eq:divergence_distortion_ineq} and the second inequality of \eqref{eq:Jacobian_neg_part_int_bound}, we get that
		\begin{multline}\label{eq:excess_estimate_part_1}
			\int_\Omega \frac{\abs{\nabla \phi}^n}{K(\mu_\phi^{+} \circ \phi)^{\gamma}}
			\le \int_\Omega \frac{\ip{\nabla \phi(x)}{\xi(x)}}{(\mu_\phi^{+} \circ \phi)^{\gamma}}
				+ \int_\Omega \frac{A}{(\mu_\phi^{+} \circ \phi)^{\gamma}}\\
			\le \norm{A}_{L^q(\Omega)} \left( \int_\Omega \frac{1}{(\mu_\phi^{+} \circ \phi)
				^{\frac{\gamma q}{q-1}}} \right)^{1 - \frac{1}{q}}.
		\end{multline}
		
		We then begin assembling the desired estimate. We first apply the superlevel Sobolev inequality, Proposition \ref{prop:distr_Sobolev_ineq}, which yields that
		\begin{equation}\label{eq:excess_estimate_part_2}
			\norm{\phi}_{L^\infty(\Omega)}^n
			\le \frac{1}{n^n \omega_n} \left( 
				\int_\Omega \frac{\abs{\nabla \phi}}{(\mu_\phi^{+} \circ \phi)^{\frac{n-1}{n}}}\right)^n
		\end{equation}
		Then we perform a 3-exponent H\"older's inequality, with exponents $n$, $np$, and $np/((n-1)p-1)$. Note that since $p^{-1} + q^{-1} < 1$, we must have $p > 1$, and since also $n \ge 2$, we have $np/((n-1)p-1) \in [1, \infty)$. Thus, we get
		\begin{align}\label{eq:excess_estimate_part_3}
			&\left( \int_\Omega \frac{\abs{\nabla \phi}}{(\mu_\phi^{+} \circ \phi)^{\frac{n-1}{n}}}\right)^n\\
			\nonumber&\quad= \left(\int_\Omega \frac{\abs{\nabla \phi}}{K^{\frac{1}{n}} 
				(\mu_\phi^{+} \circ \phi)^{\frac{\gamma}{n}}} 
			\cdot K^{\frac{1}{n}} \cdot \frac{1}{(\mu_\phi^{+} \circ \phi)^{\frac{n-1 - \gamma}{n}}} \right)^n\\
			\nonumber&\quad\le \left( \int_\Omega \frac{\abs{\nabla \phi}^n}{K (\mu_\phi^{+} \circ \phi)^{\gamma}} \right) 
			\cdot \norm{K}_{L^p(\Omega)} \cdot 
			\left( \int_\Omega \frac{1}{(\mu_\phi^{+} \circ \phi)^{\frac{p(n-1-\gamma)}{(n-1)p - 1}}} \right)^{n - 1 -\frac{1}{p}}.
		\end{align}
		
		Since $\gamma < 1$ and $n \ge 2$, we have $n - 1 - \gamma > 0$. Thus, by Lemma \ref{lem:inv_cav_funct_integral}, the final integral in \eqref{eq:excess_estimate_part_3} is finite if
		\[
			\frac{p(n-1-\gamma)}{(n-1)p - 1} < 1
			\quad \Leftrightarrow \quad
			pn - p - \gamma p < pn - p - 1
			\quad \Leftrightarrow \quad
			\gamma > p^{-1},
		\]
		which is true by our choice of $\gamma$. Moreover, Lemma \ref{lem:inv_cav_funct_integral} yields that
		\begin{multline}\label{eq:excess_estimate_part_4}
			\left( \int_\Omega \frac{1}{(\mu_\phi^{+} \circ \phi)
				^{\frac{\gamma q}{q-1}}} \right)^{1 - \frac{1}{q}}
			\left( \int_\Omega \frac{1}{(\mu_\phi^{+} \circ \phi)^{\frac{p(n-1-\gamma)}{(n-1)p - 1}}} \right)^{n - 1 -\frac{1}{p}}\\
			\le C(n,p,q,\gamma) [m_n(\Omega)]^{(1 - q^{-1} - \gamma) + ((n - 1 - p^{-1}) - (n - 1 - \gamma))}\\
			= C(n,p,q,\gamma) [m_n(\Omega)]^{1 - p^{-1} - q^{-1}}.
		\end{multline}
		Finally, noting that the choice of $\gamma$ depended only on $p$, and $q$, combining all of \eqref{eq:excess_estimate_part_1}-\eqref{eq:excess_estimate_part_4} proves the claim.
	\end{proof}
	
	Now, Theorem \ref{thm:continuity_general} is nearly immediate.
	
	\continuitygeneral*	
	\begin{proof}
		Let $D \subset \Omega$ be an open set that is compactly contained in $\Omega$. By applying Theorem \ref{thm:weakly_vanishing_L_infinity_estimate} on balls $B \subset D$ in conjunction with Lemma \ref{lem:almost_weak_mono_exp_p}, it follows that the function  $\phi$ is $\alpha$-almost weakly monotone in $D$, with $C = C(n,p,q) \norm{K}^{1/n}_{L^q(D)} \norm{A}^{1/n}_{L^q(D)}$ and $\alpha = 1 - p^{-1} - q^{-1} > 0$. Thus, by Proposition \ref{prop:weak_weak_mono_cont}, $\phi$ has a continuous representative in $D$ with the given modulus of continuity. Since the existence of continuous representatives is a local property, the claim follows.
	\end{proof}
	
	We then obtain Theorem \ref{thm:continuity} as an application of Theorem \ref{thm:continuity_general}.
	
	\continuity*
	\begin{proof}
		For every coordinate function $f_i$, we have
		\[
			\abs{\nabla f_i}^n \le \abs{Df}^n \le K J_f + \Sigma = K \left( \ip{\nabla f_i}{\xi_i} + \frac{\Sigma}{K} \right)
		\]
		a.e.\ in $\Omega$, where the weakly divergence-free vector field $\xi_i \in L^{n/(n-1)}(\Omega, \R^n)$ consists of the minors of $D(f_1, \dots, f_{i-1}, f_{i+1}, \dots, f_n)$ with the appropriate signs. Thus, Theorem \ref{thm:continuity_general} applies, and $f_i$ is continuous.
	\end{proof}
	
	\begin{rem}\label{rem:K_under_1_trick}
		Theorem \ref{thm:continuity} assumes that $K(x) \ge 1$ for all $x \in \Omega$, which is standard in the literature on mappings of finite distortion. This could be upgraded to $K(x) \ge 0$ for all $x \in \Omega$ by using the fact that Theorem \ref{thm:continuity_general} only assumes non-negativity of $K$, along with the observation made in the beginning of the proof of Theorem \ref{thm:weakly_vanishing_L_infinity_estimate} that one may re-define $K$ in $K^{-1}\{0\}$. 
		
		However, there is also a slightly better version of Theorem \ref{thm:continuity} for $K \colon \Omega \to [0, \infty)$ than what we described above. Namely, one can apply Theorem \ref{thm:continuity} when $K, \Sigma \colon \Omega \to [0, \infty)$ and $p, q \in [1, \infty]$ are such that
		\[
			K \in L^p_\loc(\Omega) \quad \text{and} \quad \frac{\Sigma}{\max(1, 2K)} \in L^{q}_\loc(\Omega)
			\quad \text{with} \quad \frac{1}{p} + \frac{1}{q} < 1.
		\]
		This is due to the following trick: \emph{if $n \ge 2$, $\Omega \subset \R^n$ is open, and $f \in W^{1,n}_\loc(\Omega, \R^n)$ satisfies \eqref{eq:generalized_finite_distortion} with $K, \Sigma \colon \Omega \to [0, \infty]$, then 
		\begin{equation}\label{eq:K_under_1_fix}
			\abs{Df(x)}^n \le [\max (1, 2K(x))] J_f(x) + 4\Sigma(x) 
		\end{equation}
		for a.e.\ $x \in \Omega$.}
		
		Indeed, suppose that \eqref{eq:generalized_finite_distortion} holds at $x \in \Omega$. If $2K(x) \ge 1$, then we may merely use the non-negativity of $\abs{Df(x)}^n$ and $\Sigma$ to conclude that
		\[
			\abs{Df(x)}^n \le 2\abs{Df(x)}^n \le 2K(x) J_f(x) + 2\Sigma(x) \le 2K(x) J_f(x) + 4\Sigma(x).
		\]
		We then consider the other case $2K(x) \le 1$. In this case, by using the estimate $\abs{J_f(x)} \le \abs{Df(x)}^n$, we obtain that 
		\[
			2\Sigma(x) \ge 2\abs{Df(x)}^n - 2K(x) J_f(x) \ge (2 - 2K(x)) \abs{Df(x)}^n \ge \abs{Df(x)}^n.
		\]
		Thus, we also have $\abs{J_f(x)} \le 2\Sigma(x)$, and consequently $J_f(x) + 2\Sigma(x) \ge 0$. It follows that
		\[
			\abs{Df(x)}^n \le 2\Sigma(x) \le J_f(x) + 4\Sigma(x),
		\]
		completing the proof of \eqref{eq:K_under_1_fix}.
	\end{rem}
	
	We conclude by briefly pointing out how Corollary \ref{cor:continuity_VFD} follows from Theorem \ref{thm:continuity}.
	
	\continuityVFD*
	\begin{proof}[Proof of Corollary \ref{cor:continuity_VFD}]
		Since $f \in W^{1,n}_\loc(\Omega, \R^n)$, we have $\abs{f - y_0} \in L^r_\loc(\Omega)$ for every $r \in [1, \infty)$ by the Sobolev embedding theorem. Thus, we have $[\abs{f - y_0}^n \Sigma]/K \in L^{q'}_\loc(\Omega)$ for some $q' \in (p/(p-1), q)$, and Theorem \ref{thm:continuity} yields the claim.
	\end{proof}
	
	\bibliographystyle{abbrv}
	\bibliography{sources}
	
\end{document}